\pdfoutput=1
\documentclass[12pt]{amsart}
\usepackage{amsmath,amssymb,amsthm,amsfonts}
\usepackage{mathrsfs}

\usepackage[colorlinks=true,allcolors=blue]{hyperref}
\usepackage[initials,nobysame,alphabetic,msc-links,backrefs]{amsrefs}
\usepackage[margin=2.5cm,heightrounded=true]{geometry}

\usepackage{enumitem}

\usepackage{latexsym}

\allowdisplaybreaks[1]
 
\numberwithin{equation}{section}
\newtheorem{theorem}{Theorem}[section]
\newtheorem{proposition}[theorem]{Proposition}
\newtheorem{lemma}[theorem]{Lemma}

\def\({\left(}
\def\){\right)}
\def\[{\left[}
\def\]{\right]}
\def\<{\langle}
\def\>{\rangle}

\newcommand{\beq}{\begin{equation}}
\newcommand{\eeq}{\end{equation}}
 
\newcommand{\la}{\lambda}

\newcommand{\UU}{\mathscr U}

\newcommand{\W}{\mathcal W}
\newcommand{\rr}{ \mathbb{R}}
\renewcommand{\L}{\mathcal L}
 
\renewcommand\bf\bfseries

\usepackage{colonequals}
\newcommand{\coloneqq}{\colonequals}
\newcommand{\eqqcolon}{\equalscolon}

\begin{document}
\title[Large conformal metrics with prescribed scalar curvature]{Large conformal metrics with prescribed scalar curvature}

\author{Angela Pistoia}
\address{{\textrm Angela Pistoia}, Dipartimento SBAI, Universit\`a di Roma ``La Sapienza", via Antonio Scarpa 16, 00161 Roma, Italy}
\email{angela.pistoia@uniroma1.it}

\author{Carlos Rom\'an}
\address{{\textrm Carlos Rom\'an}, Sorbonne Universit\'es, UPMC Univ Paris 06, CNRS, UMR 7598, Laboratoire Jacques-Louis Lions, 4, place Jussieu 75005, Paris, France}
\email{roman@ann.jussieu.fr}

\thanks{The first author has been supported by Fondi di Ateneo ``Sapienza''. The second author has been supported by a public grant overseen by the French National Research Agency (ANR) as part of the ``Investissements d'Avenir'' program (reference: ANR-10-LABX-0098, LabEx SMP)}
\date{\today}

\begin{abstract}
Let $(M,g)$ be an $n-$dimensional compact Riemannian manifold. Let $h$ be a smooth function on $M$ and assume that it has a critical  point $\xi\in M$ such that $h(\xi)=0$ and which satisfies a suitable flatness assumption. We are interested in finding conformal metrics $g_\lambda=u_\lambda^\frac4{n-2}g$, with $u>0$, whose scalar curvature is the prescribed function $h_\lambda\coloneqq \lambda^2+h$, where $\lambda$ is a small parameter.

In the positive case, i.e. when the scalar curvature $R_g$ is strictly positive, we find a family of ``bubbling'' metrics $g_\lambda$, where $u_\lambda$ blows-up at the point $\xi$ and approaches zero far from $\xi$ as $\lambda$ goes to zero.

In the general case, if in addition we assume that there exists a non-degenerate conformal metric $g_0=u_0^\frac4{n-2}g$, with $u_0>0$, whose scalar curvature is equal to $h$, then there exists a bounded family of conformal metrics $g_{0,\lambda}=u_{0,\lambda}^\frac4{n-2}g$, with $u_{0,\la}>0$, which satisfies $u_{0,\lambda}\to u_0$ uniformly as $\la\to 0$. Here, we build a second family of ``bubbling'' metrics $g_\lambda$, where $u_\lambda$ blows-up at the point $\xi$ and approaches $u_0$ far from $\xi$ as $\lambda$ goes to zero. In particular, this shows that this problem admits more than one solution.
\end{abstract}
\maketitle
%\tableofcontents

\section{Introduction}
Let $(M,g)$ be a smooth compact manifold without boundary of dimension $n\geq 3$. The prescribed scalar curvature problem (with conformal change of metric) is

{\em given a function $h$ on $M$ does there exist  a metric $\tilde g$ conformal  to $ g$ such that   the scalar curvature of  $  \tilde g $ equals $h$?}

Given a metric $\tilde g$ conformal to $g$, i.e.  $\tilde g=u^\frac4{n-2}g$ for some smooth conformal factor $u>0$, this problem is equivalent to finding a solution to 
$$
- \Delta_g u+  c(n)R_g u=c(n)hu^p,\quad u>0 \quad \mbox{on }M,
$$
where $\Delta_g=\mbox{div}_g\nabla_g$ is the Laplace-Beltrami operator, $c(n)=\frac{n-2}{4(n-1)}$, $p=\frac{n+2}{n-2}$, and $R_g$ denotes the scalar curvature associated to the metric $g$. Since it does not change the problem, we replace $c(n)h$ with $h$. We are then led to study the problem
\begin{equation}\label{p0}
- \Delta_g u+  c(n)R_g u=hu^p,\quad u>0 \quad \mbox{on }M,
\end{equation}
We suppose that $h$ is not constant, otherwise we would be in the special case of the Yamabe problem which has been completely solved in the works by Yamabe \cite{yam}, Trudinger \cite{tru}, Aubin \cite{Aubin2}, and Schoen \cite{Sch}.
For this reason we can assume in \eqref{p0} that $R_g$ is a constant.

In the book \cite{AubinBook}*{Chapter 6}, Aubin gives an exhaustive description of known results. Next, we briefly recall some of them.
\begin{itemize}[font=\normalfont,leftmargin=*]
\item{\bf The negative case, i.e. $R_g<0$.}

\noindent
A necessary condition for existence is that $\int\limits_M h d\nu_g<0$ (a more general result can be found in \cite{kawa}).

When $h<0$, \eqref{p0} has a unique solution (see for instance \cites{kawa,Aubin2}). The situation turns out to be more complicated when $h$ vanishes somewhere on $M$ or if it changes sign. When $\max_M h=0$, Kazdan and Warner \cite{kawa}, Ouyang \cite{ouy}, V\'azquez and V\'eron \cite{vave}, and del Pino \cite{delpino}  proved the existence of a unique solution, provided that a lower bound on $R_g$, depending on the zero set of $h$, is satisfied. 
The general case was studied by Rauzy in \cite{rau1}, who extended the previous results to the case when $h$ changes sign. Letting $h^-\coloneqq\min\{h,0\}$ and $h^+\coloneqq\max\{h,0\}$, the theorem proved in \cite{rau1} reads as follows.

\begin{theorem} \label{zero-} Let $\mathcal A\coloneqq\left\{u\in H^1_g(M)\ |\ u\ge0,\ u\not\equiv0,\ \int\limits_M h^-ud\nu_g=0\right\}$ and 
$$
\Lambda_0\coloneqq\inf\limits_{u\in\mathcal A}\frac{\int\limits_M|\nabla u|^2d\nu_g}{\int\limits_M u^2d\nu_g},
$$
with $\Lambda_0=+\infty $ if $\mathcal A=\emptyset$.
There exists a constant $C(h)>0$, depending only on $\frac{\min_Mh^-}{\int\limits_M h^-d\nu_g}$, such that if
\begin{equation}\label{ass-rau}-c(n)R_g<\Lambda_0\quad \hbox{and}\quad \frac{\max_Mh^+}{\int\limits_M| h^-|d\nu_g}<C(h)\end{equation}
then \eqref{p0} has a solution.
\end{theorem}
The dependence of the constant $C(h)$ on the function $h^-$ can be found in \cite{aubis}.
An interesting feature is that if $h$ changes sign then the uniqueness is not true anymore, as showed by Rauzy in \cite{rau2}.

\begin{theorem}\label{teo-rau2}
Assume \eqref{ass-rau} and let $\xi\in M$ such that $h(\xi)=\max\limits_M h>0$. If one of the following conditions hold:
\begin{enumerate}[leftmargin=*]
\item $6\le n\le9$ and $\Delta_g h(\xi)=0$;
\item $n\ge 10,$ the manifold is not locally conformally flat, and $\Delta_g h(\xi)=\Delta_g ^2h(\xi)=0$;
\end{enumerate}
then \eqref{p0} admits at least two distinct  solutions.
\end{theorem} 
\item{\bf The zero case, i.e. $R_g=0$.}

\noindent 
Necessary conditions for existence are that $h$ changes sign and  $\int\limits_M h d\nu_g<0.$
Some of the existence results proved by Escobar and Schoen \cite{escobar}, Aubin and Hebey \cite{ah}, and Bismuth \cite{bis} can be summarized as follows.
\begin{theorem}\label{zero}
Let $\xi\in M$ such that $h(\xi)=\max\limits_M h>0.$
If one of the following conditions hold:
\begin{enumerate}[leftmargin=*]
\item $3\leq n \leq 5$ and all the derivatives of $h$ at the point $\xi$ up to order $n-3$ vanish;
\item $(M,g)$ is locally conformally flat, $n\ge 6$  and all the derivatives of $h$ at the point $\xi$ up to order $n - 3$ vanish;
\item the Weyl's tensor at $\xi$ does not vanish, [$n=6$ and $\Delta_g h(\xi)=0$], or [$n\ge7$ and $\Delta_g h(\xi)=\Delta^2_g h(\xi)=0$];
\end{enumerate}
then \eqref{p0} has a solution
\end{theorem}

\item {\bf The positive case, i.e. $R_g>0$.}

\noindent
A necessary condition for existence is that $\max_M h>0.$  
Some of the existence results proved by Escobar and Schoen \cite{escobar}, Aubin and Hebey \cite{ah}, and Hebey and Vaugon \cite{hv2}  can be summarized as follows.
\begin{theorem}\label{zero+}
Assume that $(M,g)$ is not conformal to the standard sphere $(\mathbb S^n, g_0)$.
Let $\xi\in M$ such that $h(\xi)=\max\limits_M h>0.$
If one of the following conditions hold:
\begin{enumerate}[leftmargin=*]
\item $n=3$ or [$n\geq 4$, $(M,g)$ is locally conformally flat, and all the derivatives of $h$ at the point $\xi$ up to order $n -2$ vanish];
\item the Weyl's tensor at $\xi$ does not vanish, [$n=6$ and $\Delta_g h(\xi)=0$], or [$n\ge7$ and $\Delta_g h(\xi)=\Delta^2_g h(\xi)=0$];
\end{enumerate}
then \eqref{p0} has a solution.
\end{theorem}
The prescribed scalar curvature problem on the standard sphere has also been largely studied. We refer the interested reader to \cites{hv1,yyl2,yyl3}.
\end{itemize}

\medskip
In the rest of this paper, we focus our attention on the case $h_\lambda(x)\coloneqq\lambda^2+h(x)$ where $h\in C^2(M)$ and $\la>0$ is a small parameter. Namely, we study the problem
\beq\label{scp}
-\Delta_g u+  c(n)R_g u=(\la^2+h)u^p ,\ u>0 \quad \mbox{on }M.
\eeq
In order to state our main results, let us introduce two assumptions. In our first theorem we assume that $h$ satisfies the following \emph{global} condition:
\begin{align}\label{nd1}
\hbox{\em there exists a non-degenerate solution $u_0$ to} \nonumber\\
-\Delta _g u_0+c(n)R_g u _0=h u_0^p,\quad u_0>0 \quad \mbox{on }M.
\end{align}
The existence of a solution to \eqref{nd1} is guaranteed if $h$ is as in Theorems \ref{zero-}, \ref{zero}, or \ref{zero+}.
The non-degeneracy condition is a delicate issue and it is discussed in Subsection \ref{nondege}. {\em It would be interesting to see if for ``generic'' functions $h$ the solutions of \eqref{nd1} are non-degenerate.} 

Under this assumption, it is clear that if $\lambda$ is small enough then \eqref{scp} has a solution $u_{0,\la}\in C^2(M)$ such that $\|u_{0,\la}-u_0\|_{C^2(M)}\to0$ as $\la\to0.$

In addition, the following \emph{local} condition is assumed in both of our results:
\begin{gather}
\hbox{\em There exist a point $\xi\in M$ and some real numbers $\gamma \ge 2$, $a_1,\dots,a_n\not=0$, with $\sum\limits_{i=1}^n a_i>0$,} \nonumber\\
\hbox{\em such that, in some geodesic normal coordinate system centered at $\xi$, we have }\nonumber\\
h(y)=-\sum\limits_{i=1}^na_i|y_i|^{\gamma} +R(y) \quad \mbox{if } y\in B(0,r),\ \hbox{for some $r>0$},\label{h}\\
\hbox{where $R$ satisfies $\lim\limits_{y\to0} {R(y) |y|^{-\gamma} }=0.$}\nonumber
\end{gather}
In particular, $h(\xi)=\nabla h(\xi)=0$. The number $\gamma$ is called {\em the order of flatness} of $h$ at the point $\xi.$ 
Observe that $\xi$ cannot be a minimum point and that if all the $ a_i$'s are positive then $\xi$ is a local maximum point of $h$.

\medskip
Let us now introduce the standard $n$-dimensional bubbles, which are defined via
\begin{align*}
U_{\mu,y}(x)=\mu^{-\frac{n-2}2}U\left(\frac{x-y}\mu\right),\ \mu>0,\ y\in \rr^n,\\ 
\hbox{where}\ U(x)=\alpha_n\frac{1}{(1+|x|^2) ^{\frac{n-2}2}},\  \alpha_n=[n(n-2)]^{\frac{n-2}4}.
\end{align*}
These functions are all the positive solutions to the critical problem (see \cites{Aubin1,Talenti}) 
$$
-\Delta U= U^p\quad \mbox{in }\rr^n.
$$
Our first result concerns the multiplicity of solutions to problem \eqref{scp}.
\begin{theorem}\label{main0}
Assume that $(M,g)$ is not conformal to the standard sphere $(\mathbb S^n, g_0)$, \eqref{nd1}, and \eqref{h}. If one of the following conditions hold:
\begin{enumerate}[leftmargin=*]
\item $3\le n\le 5$ and $\xi$ is a non-degenerate critical point of $h$, i.e. $\gamma=2$; 
\item $n=6$ and $\gamma\in (2,4)$; 
\item $7\le n\le9$ and $\gamma=4$;
\item $n\ge10$, $(M,g)$ is locally conformally flat , and $\gamma\in\left(\frac{n-2}2, \frac n2\right)$; 
\item[(5)] $n\ge10$, the Weyl's tensor at $\xi$ does not vanish, and $\gamma\in(4,4+\epsilon)$ for some $\epsilon>0$;
\end{enumerate}
then, provided $\lambda$ is small enough, there exists a solution $u_\la$ to problem \eqref{scp} which blows-up at the point $\xi$ as $\lambda\to0$. Moreover, as $\la \to 0$, we have
$$
\left\|u_\la(x)- u_0(x) -\la^{-\frac{n-2}2}{\mu_\la}^{-\frac{n-2}2}U\left(\frac{d_g(x,\xi_\la)}{\mu_\la}\right)\right\|_{H^1_g(M)}\to0,$$
where the concentration point $\xi_\la\to\xi$ and the concentration parameter $\mu_\la\to0$ with a suitable rate with respect to $\la$, which depends on the order of flatness $\gamma$ {\normalfont(see \eqref{mu1}, \eqref{mu2}, \eqref{mu9}, \eqref{mu3})}.

Finally, if $h\in C^\infty(M)$, then $\lambda^2+h$ is the scalar curvature of a metric conformal to $g.$
\end{theorem}
This is the first multiplicity result in the zero and positive cases. In the negative case, it extends the results of Theorem \ref{teo-rau2} to locally conformally flat manifolds, to low-dimensional manifolds (i.e. $3\le n\le5$), to higher-dimensional manifolds (i.e. $6\le n\le9$) when the order of flatness at the maximum point $\xi$ is at least 2, and to non-locally conformally manifolds when $n\geq 10$ and the order of flatness at the maximum point $\xi$ is at least $4$. Moreover, it also provides an accurate description of the profile of the solution as $\la$ approaches zero.

\medskip
Our second result concerns the existence of solutions to problem \eqref{scp} in the positive case, without need of the global condition \eqref{nd1} on $h$. 
\begin{theorem}\label{main1}
Assume that $(M,g)$ is not conformal to the standard sphere $(\mathbb S^n, g_0)$, $R_g>0$, and \eqref{h}. If one of the following conditions hold:
\begin{enumerate}[leftmargin=*]
\item $n=3,4,5$ or [$n\ge6$ and $(M,g)$ is locally conformally flat] and $\gamma\in(n-2,n)$;
\item $n\ge 6$, the Weyl's tensor at $\xi$ does not vanish, and $\gamma\in(4,n)$;
\end{enumerate}
then, provided $\lambda$ is small enough, there exists a solution $u_\la$ to problem \eqref{scp} which blows-up at the point $\xi$ as $\lambda\to0$. Moreover, as $\la \to 0$, we have
$$
\left\|u_\la(x)-\la^{-\frac{n-2}2}{\mu_\la}^{-\frac{n-2}2}U\left(\frac{d_g(x,\xi_\la)}{\mu_\la}\right)\right\|_{H^1_g(M)}\to0,$$
where the concentration point $\xi_\la\to\xi$ and the concentration parameter $\mu_\la\to0$  with a suitable rate with respect to $\la$, which depends on the order of flatness $\gamma$  {\normalfont(see \eqref{Eq8})}.

Finally, if $h\in C^\infty(M)$, then $\lambda^2+h$ is the scalar curvature of a metric conformal to $g.$
\end{theorem}

Our results have been inspired by the recent papers by Borer, Galimberti, and Struwe \cite{bgs} and del Pino and Rom\'an \cite{pinrom}, where the authors studied the prescribed Gauss curvature problem on a surface of dimension $2$ in the negative case. In particular, they built large conformal metrics with prescribed
Gauss curvature $\kappa$, which exhibit a bubbling behavior around maximum points of $\kappa$ at zero level. We also refer the reader to \cites{dp2,dp3}, where equations closely related to \eqref{p0} have been treated.

The proof of our results relies on a Lyapunov-Schmidt procedure (see for instance \cites{blr,dp1}). To prove Theorem \ref{main0}, we look for solutions to \eqref{scp} which share a suitable bubbling profile close to the point $\xi$ and the profile of the solution to the unperturbed problem \eqref{nd1} far from the point $\xi$. The accurate description of the ansatz is given in Section \ref{sec1}, which also contains a non-degeneracy result. 
The finite dimensional reduction is performed in Section \ref{sec2}, which also includes the proof of Theorem \ref{main0}. All the technical estimates are postponed in the Appendix.
In Section \ref{sec3} we prove Theorem \ref{main1}, which can be easily deduced by combining the results
proved in Section \ref{sec2} and in the Appendix with some recent results obtained by Esposito, Pistoia, and V\'etois in \cite{espive1}.

\subsection*{Acknowledgments}
We would like to thank Manuel del Pino for bringing this problem to our attention, as well as for many helpful discussions we had with him. We also thank the anonymous referee for his or her careful reading and very useful comments.

\section{The approximated solution}\label{sec1}
\subsection{The ansatz}
To build the approximated solution close to the point $\xi$ we use some ideas introduced in \cites{espi,rove}. The main order of the approximated solution close to the point $\xi$ looks like the bubble
\beq\label{ula}
{\la^{-\frac{n-2}{2}}\mathcal U_{t,\tau}}(x)\coloneqq\la^{-\frac{n-2}{2}}\mu^{-\frac{n-2}{2}}U\(\frac{ \exp_\xi^{-1}\(x\)}{\mu}-\tau\)\quad \hbox{if }d_g(x,\xi)\le r,
\eeq
where the point $\tau\in \rr^n $  depends on $\lambda$ and the parameter $\mu=\mu_\la(t)$  satisfies
\beq \label{parameters}
\mu=t\lambda^{\beta }\quad \hbox{for some}\ t>0 \ \hbox{and } \beta >1.
\eeq
The choice of $\beta$ depends on $n$, on the geometry of the manifold at the point $\xi$, i.e. the Weyl's tensor at $\xi$ and on the order of flatness of the function $h$ at $\xi$, i.e.  the number $\gamma\coloneqq2+\alpha$ in \eqref{h}.

Let us be more precise. Let $r\in(0,i_g(M))$ be fixed, where $i_g(M)$ is the injectivity radius of $(M,g)$, which is strictly positive since the manifold is compact. Let $\chi\in C^\infty(\mathbb R)$ be a cut-off function such that $0\le\chi\le 1$ in $\mathbb R$, $\chi =1$ in $[-r/2,r/2]$ and $\chi =0$ in $\mathbb R\setminus (-r,r)$.
We denote by $d_g$ the geodesic distance in $(M,g)$ and by $\exp_\xi^{-1}$ the associated geodesic coordinate system. 
We look for solutions of \eqref{scp} of the form
\beq\label{ansatz}
u_\la(x)=u_0(x)+\la^{-\frac{n-2}{2}}\W_{t,\tau}(x)+\phi_\la(x),	
\eeq
where the definition of the blowing-up term $\W_{t,\tau}$ depends on the dimension of the manifold and also on its geometric properties. The higher order term $\phi_\la$ belongs to a suitable space which will be introduced in the next section.
More precisely, $\W_{t,\tau}$ is defined in three different ways:

\begin{itemize}[font=\normalfont,leftmargin=*]
\item{\bf The case $n=3,4,5$.}

\noindent
It is enough to assume
\beq\label{w1}
\W_{t,\tau}(x)=\chi\(d_g (x,\xi ) \){\mathcal U_{t,\tau}}(x),	
\eeq
where ${\mathcal U_{t,\tau}}$ is defined in \eqref{ula}.
The concentration parameter $\mu$ satisfies
\beq
\label{mu1}\mu=t\lambda^\frac{n+2}{2\alpha-n+6},\quad \mbox{with } t>0\ \mbox{provided }0\le\alpha<n-2.
\eeq

\item{\bf The cases $n\ge10$ when $\textnormal{Weyl}_g(\xi)$ is non-zero and $6\le n\le 9$.}

\noindent
It is necessary to correct the bubble ${\mathcal U}_{t,\tau}$ defined in \eqref{ula} by adding a higher order term as in \cite{espi}, namely
\beq\label{w2}
\W_{t,\tau}(x)=\chi\(d_g (x,\xi ) \)\({\mathcal U_{t,\tau}}(x)+\mu^2{\mathcal V_{t,\tau}}(x)\)
\eeq
where 
$$
{\mathcal V_{t,\tau}}(x)=\mu^{-\frac{n-2}{2}}V\(\frac{ \exp_\xi^{-1}\(x\)}{\mu}-\tau\)\quad \mbox{if }d_g(x,\xi)\le r.
$$
The choice of parameter $\mu$  depends on $n$. More precisely
if $n\ge10$ and the Weyl's tensor at $\xi$ is non-zero we choose
\beq
\label{mu2}\mu=t\lambda^\frac{2}{\alpha-2},\quad
\mbox{with } t>0\ \hbox{provided}\ 2<\alpha<\frac{2n}{n-2}.
\eeq 
If $6\le n\le 9$ we choose
\beq\label{mu9}
\mu=t\lambda^\frac{n+2}{2\alpha-n+6},\quad
\mbox{with}\ t>0\ \hbox{provided}\ \frac{n-6}{2}<\alpha<\min\left\{\frac{16}{ n-2},\frac{n^2-6n+16}{2(n-2)}\right\}.
\eeq  

The function $V$ is defined as follows.
If we write $u(x)=u\({ \exp_\xi^{-1}\(x\) } \)$ for $x\in B_g(\xi,r)$ and $y=\exp_\xi \(x\)\in B(0,r)$, then a comparison between the conformal Laplacian $\L_g=-\Delta_g +c(n)R_g$ with the euclidean Laplacian shows that there is an error, which at main order looks like
\beq\label{conf-eu} 
\mathcal L_g u +\Delta u\sim  +\frac 13 \sum\limits_{a,b,i,j=1}^nR_{iab j}(\xi)y_a y_b\partial^2_{ij}u +\sum\limits_{i, l,k=1}^n \partial_l \Gamma^k_{ii}(\xi)y_l\partial_k u  +{c(n) R_g(\xi)}u.
\eeq
Here $ R_{iab j}$ denotes the Riemann curvature tensor, $\Gamma^k_{ij}$ the Christoffel's symbols and $R_g$ the scalar curvature. This easily follows by standard properties of the exponential map, which imply 
$$
-\Delta_g u = -\Delta u -(g^{ij}-\delta^{ij})\partial^2_{ij}u+g^{ij}\Gamma^k_{ij}\partial_k u, 
$$
with
$$
g^{ij}(y)=\delta^{ij}(y) -\frac 13 R_{iab j}(\xi)y_a y_b +O(|y|^3)\ \hbox{and}\ 
g^{ij}(y)\Gamma^k_{ij}(y)=\partial_l \Gamma^k_{ii}(\xi)y_l +O(|y|^2).
$$
To build our solution it shall be necessary to kill the R.H.S of \eqref{conf-eu} by adding to the bubble a higher order term $V$ whose existence has been established in \cite{espi}. To be more precise, we need to remind (see \cite{bieg}) that all the solutions to the linear problem
$$
-\Delta v=pU^{p-1}v\quad\hbox{in}\ \mathbb{R}^n,
$$
are linear combinations of the functions
\beq\label{Eq14}
Z_0\(x\)=x\cdot\nabla U(x)+\frac{n-2}{2} U(x),
\quad
Z_i\(x\)=\partial_i U(x), \ i=1,\dots,n.
\eeq
The correction term $V$ is built in the following Proposition (see  Section 2.2 in \cite{espi}).
\begin{proposition} %\label{bubbleV}
There exist $\nu(\xi)\in\mathbb R$ and a function $V\in   \mathcal{D}^{1,2}(\mathbb{R}^n)$ solution to
\begin{align*} %\label{EqV1}
-\Delta  V &-f'(U) V = \\& -\sum\limits_{a,b,i,j=1}^n\frac13 R_{iabj} (\xi)y_ay_b\partial^2_{ij} U
-\sum\limits_{i, l,k=1}^n \partial_l\Gamma^k_{ii}(\xi)y_l\partial_kU -c(n)R_g(\xi) U+\nu(\xi)Z_0,
\end{align*}
in $\rr^n$, with
$$\displaystyle\int\limits_{\mathbb{R}^n}V(y)Z^i(y)dy=0,\quad  i=0,1,\dots,n $$
and
$$
|V(y)|+ |y|\left|\partial_kV(y)\right|+|y|^2 \left|\partial^2_{ij}V (y)\right| =O\(\frac{1}{ (1+|y|^2)^\frac{n-4}{2}}\),\quad y\in\mathbb{R}^n. 
$$
\end{proposition}
 
\item {\bf The case $n\ge10$ when $(M,g)$ is locally conformally flat}.

\noindent 
In this case it is necessary to perform a conformal change of metric as in \cite{rove}. Indeed, there exists a function $\Lambda_\xi\in C^\infty(M)$ such that the conformal metric $g_\xi=\Lambda_\xi^\frac{4}{(n-2)}g$ is flat in $B_g(\xi,r)$. The metric can be chosen so that $\Lambda_\xi(\xi)=1$. Then, we choose
\beq\label{w3}
\W_{t,\tau}(x)=\chi\(d_{g_\xi} (x,\xi ) \)\Lambda_\xi(x){\mathcal U_{t,\tau}}(x),	
\eeq
where ${\mathcal U_{t,\tau}}$ is defined in \eqref{ula} and the exponential map is taken with respect to the new metric $g_\xi$. In this case, the concentration parameter $\mu$ satisfies 
\beq\label{mu3}
\mu=t\lambda^\frac{n+2}{2\alpha-n+6},\quad
\mbox{with } t>0\ \hbox{provided}\ \frac{n-6}{2}<\alpha<\frac{n^2-6n+16}{2(n-2)}.
\eeq  

\end{itemize}

\subsection{The higher order term}
Let us consider the Sobolev space $H_g^1(M)$ equipped with the scalar product  
$$
(u,v)=\int_M\(\< \nabla_g u,\nabla_g v\>_g+uv\)d\nu_g,
$$
and let $\|\cdot\|$ be the induced norm.
Let us introduce the space where the higher order term $\phi_\la$ in \eqref{ansatz} belongs to. Let $Z_0,Z_1,\dots, Z_n$ be the functions introduced in \eqref{Eq14}.
We define
$$
Z_{i,t,\tau}(x) =\mu^{-\frac{n-2}2} \chi (d_{g_\xi}(x,\xi))\Lambda_\xi(x)Z_i\left(\frac{\exp_\xi^{-1}(x)}{\mu}-\tau \right)\quad i=0,1,\dots,n,
$$
where $g_\xi$ and $\Lambda_\xi$ are defined as in \eqref{w3}   and we also agree that $g_\xi\equiv g$, $\Lambda_\xi(x)\equiv 1$ if the ansatz is \eqref{w1} or \eqref{w2}.
Therefore, $\phi_\la\in H^\perp$ where
$$ 
H^\perp\coloneqq\left\{\phi\in H^1_g(M)\ :\ \int_M\phi Z_{i,t,\tau}d\nu_g=0\ \hbox{for any}\ i=0,1,\dots,n\right\}.
$$

\subsection{A non-degeneracy result} \label{nondege}
When the solution $u_0$ of \eqref{nd1} is a minimum point of the energy functional naturally associated with the problem, the non-degeneracy is not difficult to obtain as showed in Lemma \ref{nonde}.
In the general case $u_0$ is a critical point of the energy of a min-max type and so the non-degeneracy is a more delicate issue. As far as we know there are no results in this direction.

\begin{lemma}\label{nonde} Assume $R_g<0,$ $\max_M h=0$, and the set $\{x\in M\ |\ h(x)=0\}$ has empty interior set. Then the unique solution 
$u_0$   to \eqref{nd1}  is non-degenerate, i.e. the linear problem
$$
-\Delta_g \psi + c(n)R_g \psi-ph(x)u_0^{p-1}\psi=0 \quad \mbox{on }M, 
$$
admits only the trivial solution.
\end{lemma}

\begin{proof} 
Del Pino in \cite{delpino} proved that problem \eqref{nd1} has a unique solution, which    is a   minimum point of the energy functional
$$
J(u)\coloneqq\frac{1}{2}\int_M\left(|\nabla_g u|^2_g +c(n)R_gu^2 \right) d\nu_g-\frac1{p+1}\int_M h|u|^{p+1}d\nu_g.
$$
Therefore the quadratic form
$$
D^2 J(u_0)[\phi,\phi]= \int_M\left(|\nabla_g \phi|^2_g +c(n)R_g\phi^2-phu_0^{p-1}\phi^2\right)d\nu_g,\quad \phi\in H^1_g(M)
$$
is positive definite. In particular, the problem
\beq\label{nd2}
-\Delta _g \phi_i+c(n)R_g \phi_i-p h u_0^{p-1}\phi_i=\lambda_i\phi_i \quad \mbox{on }M,
\eeq
has a non-negative first eigenvalue $\lambda_1$ with associated eigenfunction $\phi_1>0$ on $M$.

\medskip
If $\lambda_1=0$ then we test \eqref{nd1} against $\phi_1$ and \eqref{nd2} against $u_0$, we subtract and we get
$$
(p-1)\int_M hu_0^p\phi_1d\nu_g=0,
$$
which gives a contradiction because $h\not=0$ a.e. in $M$ and $u_0>0$ on $M$.
\end{proof}

\section{The finite dimensional reduction}\label{sec2}
We are going to solve problem
\beq\label{scp1}
\mathcal L_g u=\(\la^2 +h\)f(u) \quad \mbox{on }M, 
\eeq
where $\mathcal L_g$ is the conformal Laplacian and $f(u)=(u^+)^p$, $u^+(x)\coloneqq\max\{u(x),0\}$, using a Ljapunov-Schmidt procedure. We rewrite \eqref{scp1} as
\beq\label{LyapSch}
L(\phi_\la)=-E+(\la^2+h)N(\phi_\la)\quad \mbox{on }M,
\eeq
where setting
\beq\label{firstaprox}
\UU_{\la}(x)=\UU_{\la,t,\tau}(x)\coloneqq \la^{-\frac{n-2}2}\W_{t,\tau}(x)+u_0(x),
\eeq
the linear operator $L(\cdot)$ is defined by
\beq\label{elle}
L(\phi)\coloneqq\mathcal L_g\phi-(\la^2 +h)f'(\UU_\la)\phi,
\eeq
the error term is defined by
\beq\label{error}
E\coloneqq\mathcal L_g \UU_\la-(\la^2+h)f(\UU_\la)
\eeq
and the higher order term $N(\cdot)$ is defined by
$$
N(\phi)\coloneqq f\(\UU_\la+\phi\)-f\(\UU_\la\)-f'\(\UU_\la\)\phi.
$$

First of all, it is necessary to estimate the error term $E$.
\begin{proposition}\label{errore}
Let $a,b\in \rr_+$ be such that $0<a<b$ and $K$ be a compact set in $\rr^n.$ There exist positive numbers $\lambda_0,$ $C$ and $\epsilon>0$ such that for any $\lambda\in(0,\lambda_0)$, for any $t\in[a,b]$ and for any point $\tau\in K$ we have 
$$
\|E\| _{L^{\frac{2n}{n+2}}(M)}\le C \la^{\frac{2(n+2)-\alpha(n-2)}{ 2}+\epsilon}\quad \hbox{if \eqref{mu2} holds}
$$
or
$$
\|E\| _{L^{\frac{2n}{n+2}}(M)}\le C\la^{\frac{(n-2-\alpha)(n-2)}{2(2\alpha-n+6)}+\epsilon}\quad \hbox{if \eqref{mu1} or \eqref{mu9} or \eqref{mu3} hold.}
$$
\end{proposition} 
\begin{proof} The proof is postponed in Subsection \ref{errore-proof}.\end{proof}

Then, we develop a solvability theory for the linearized operator $L$ defined in \eqref{elle} under suitable orthogonality conditions.   
\begin{proposition}\label{PropLinearTheory}
Let $a,b\in \rr_+$ be fixed numbers such that $0<a<b$ and $K$ be a compact set in $\rr^n.$ There exist positive numbers $\lambda_0$ and $C$, such that for any $\lambda\in(0,\lambda_0)$, for any $t\in[a,b]$ and for any point $\tau\in K$, given  $\ell\in L^{\frac{2n}{n+2}}(M)$ there is a unique function $\phi_\la=\phi_{\la,t,\tau}(\ell)$ and unique scalars $c_i$, $i=0,\dots,n$ which solve the linear problem 
\beq\label{LinearTheory}
\left \lbrace
\begin{array}{rcll}
L(\phi)&=&\displaystyle \ell+\sum_{i=0}^nc_iZ_{i,t,\tau}&\mbox{on }M,\\
\displaystyle \int_M \phi Z_{i,t,\tau}d\nu_g&=& 0, &\mbox{for all }i=0,\dots,n.
\end{array}\right.
\eeq 
Moreover,
\beq\label{LinTheEstimate}
\displaystyle \|\phi_\la\|_{H_g^1(M)}\leq C \|\ell\|_{ L^{\frac{2n}{n+2}}(M)}.
\eeq
\end{proposition}
\begin{proof} The proof is postponed in Subsection \ref{lin-proof}.\end{proof}

Next, we reduce the problem to a finite-dimensional one by solving a non-linear problem.
\begin{proposition}\label{IntermediateProp}
Let $a,b\in \rr_+$ be fixed numbers such that $0<a<b$ and $K$ be a compact set in $\rr^n.$ There exist positive numbers $\lambda_0$ and $C$, such that for any $\lambda\in(0,\lambda_0)$, for any $t\in[a,b]$ and for any point $\tau\in K$, there is a unique function $\phi_\la=\phi_{\la,t,\tau}$ and unique scalars $c_i$, $i=0,\dots,n$ which solve the non-linear problem 
\beq\label{IntermediateProblem}
\left \lbrace
\begin{array}{rcll}
L(\phi)&=&\displaystyle -E+(\lambda^2+h)N(\phi)+\sum_{i=0}^nc_iZ_{i,t,\tau}&\mbox{on }M,\\
\displaystyle \int_M \phi Z_{i,t,\tau}d\nu_g&=& 0, &\mbox{for all }i=0,\dots,n.
\end{array}\right.
\eeq 
Moreover,
\beq\label{estimatePhiError}
\|\phi_\la\|_{H_g^1(M)}\leq C \|E\|_{L^{\frac{2n}{n+2}}(M)} 
\eeq
and $\phi_\la$ is continuously differentiable with respect to $t$ and $\tau.$
\end{proposition}
\begin{proof} The proof relies on standard arguments (see \cite{espive1}).\end{proof}

After Problem \eqref{IntermediateProblem} has been solved, we find a solution to Problem \eqref{LyapSch} if we manage to adjust $(t,\tau)$ in such a way that 
\beq\label{eqcoef}
c_i(t,\tau)=0,\quad \mbox{for all }i=0,\dots,n.
\eeq
This problem is indeed variational: it is equivalent to finding critical points of a function of $t,\tau$. To see that let us introduce the energy functional $J_\la$ defined on $H_g^1(M)$ by
$$%\label{energy}
J_\la(u)= \int_M \(\frac12|\nabla_g u|_g^2+\frac12 c(n)R_gu^2-\frac{\la^2}{p+1} (u^+)^pu+\frac1{p+1} h(u^+)^pu\)d\nu_g.
$$
An important fact is that the positive critical points of $J_\la$ are solutions to \eqref{scp}. For any number $t>0$ and any point $\tau\in \mathbb R^n$, we define the reduced energy
\begin{equation}\label{reduced-energy}
\mathcal{J}_\la(t,\tau)\coloneqq J_\la(\UU_\la +\phi_\la),
\end{equation}
where $\UU_\la=\UU_{\la,t,\tau}$ is as in \eqref{firstaprox} and $\phi_\la=\phi_{\la,t,\tau}$ is given by Proposition \ref{IntermediateProp}. Critical points of $\mathcal J_\la$ correspond to solutions of \eqref{eqcoef} for small $\la$, as the following result states.

\begin{lemma}\label{ridotto} The following properties hold:
\begin{enumerate}[leftmargin=*]
\item[(I)] There exists $\lambda_0>0$ such that for any $\lambda\in(0,\lambda_0)$ if $(t_\la,\tau_\la)$ is a critical point of $\mathcal{J}_\la$ then the function $u_\la=\UU_\la +\phi_{\la,t_\la,\tau_\la}$ is a solution to \eqref{scp1}.
\item[(II)] Let $a,b\in \rr_+$ be fixed numbers such that $0<a<b$ and let $K$ be a compact set in $\rr^n.$ There exists $\lambda_0>0$ such that, for any $\lambda\in(0,\lambda_0)$, we have:
\begin{itemize}
\item[(a)] if $n\ge10$, $\textnormal{Weyl}_g(\xi)\not=0$, and \eqref{mu2} holds then  
\begin{equation}\label{rido1}
\mathcal{J}_\la(t,\tau)=A_0- \la^{\frac{2(n+2)-\alpha(n-2)}{ \alpha-2}}\left(\underbrace{\[A_1|\textnormal{Weyl}_g(\xi)|_g^2t^4-A_2t^{2+\alpha}\sum\limits_{i=1}^n\int\limits_{\rr^n}a_i \frac{|y_i+\tau_i|^{2+\alpha}}{(1+|y|^2)^n}dy\]}_{\Theta_1(t,\tau)}+o(1)\right),
\end{equation} 
$C^1-$uniformly with respect to $t\in[a,b]$ and $\tau\in K$;
\item[(b)] if one of the following conditions is satisfied:
\begin{itemize}
\item[(i)] $3\le n\le 5$ and \eqref{mu1} holds;
\item[(ii)] $6\le n\le9$ and \eqref{mu9} holds;
\item[(iii)] $n\ge10$, $(M,g)$ is locally conformally flat, and \eqref{mu3} holds;
\end{itemize}
then   
\begin{equation}\label{rido2}
\mathcal{J}_\la(t,\tau)=A_0- \la^\frac{(n-2-\alpha)(n-2)}{2\alpha-n+6 }\left(\underbrace{\[A_3u_0(\xi)t^\frac{n-2}{2}-A_2t^{2+\alpha}\sum\limits_{i=1}^n\int\limits_{\rr^n}a_i \frac{|y_i+\tau_i|^{2+\alpha}}{(1+|y|^2)^n}dy\]}_{\Theta _2(t,\tau)}+o(1)\right),
\end{equation}
$C^1-$uniformly with respect to $t\in[a,b]$ and $\tau\in K$.
\end{itemize}
Here, $A_1,$ $A_2$, and $A_3$ are constants only depending on $n$ and
\beq\label{a0}
\begin{split}
A_0\coloneqq\int\limits_M\(\frac12|\nabla_g u_0|_g^2+\frac12 c(n)R_gu_0^2-\frac{\la^2}{p+1}  u_0^{p+1} \right. & \left. +\frac1{p+1} hu_0^{p+1}\)d\nu \\
& +\int\limits_{\mathbb R^n}\(\frac 12|\nabla U|^2-\frac 1{p+1}U^{p+1}\)dy.
\end{split}
\eeq
\end{enumerate}
\end{lemma}
\begin{proof}The proof of (I) is standard (see \cite{espive1}). The proof of (II)  is postponed in Subsection \ref{ridotto-proof}.\end{proof}

The next result is essential to find solutions to \eqref{scp}.
\begin{lemma}\label{positivo}
There exists $\lambda_0>0$ such that for any $\lambda\in(0,\lambda_0)$ if $(t_\la,\tau_\la)$ is a critical point of $\mathcal{J}_\la$ then
the function $u_\la=\UU_\la +\phi_{\la,t_\la,\tau_\la}$ is a classical solution to \eqref{scp}.
\end{lemma}
\begin{proof} By Lemma \ref{ridotto} we deduce that $u_\la$ solves \eqref{scp1}. Arguing as in Appendix B of \cite{struwe}, one easily sees that $u_\lambda\in C^2(M)$.

\medskip 
It only remains to   prove that $u_\la>0$ on $M$.  This is immediate in the positive case, i.e. $R_g>0$, because the maximum principle holds. Let us consider the case $R_g\le0.$

\noindent 

We consider the set $\Omega_\la\coloneqq\{x\in M\ | \(u_\lambda-\lambda\)^-(x)<0\}$. Let $m_0\coloneqq\min_Mu_0>0.$ By the definition of $u_\la$   we immediately get that for all $\lambda$ sufficiently small $\phi_\lambda <-\frac{m_0}2$ in $\Omega_\lambda$. Thus, since $\phi_\lambda \to 0$ in $L^2(M)$, we deduce $|\Omega_\lambda|\to 0$ as $\la \to 0$. 
Now, set $v \coloneqq u_\lambda-\lambda.$

Testing \eqref{scp1} against $v^-$ we get
\begin{align*}
\int_{\Omega_\la}| \nabla_g v^-|_g^2 d\nu_g+  c(n)R_g & \int_{\Omega_\la}(v^-)^2d\nu_g -\int_{0< u_\la < \la}(\la^2+h)(u_\la^+)^{p-1}(v^-)^2d\nu_g\\
& +\la\[c(n)R_g\int_{\Omega_\la}v^-d\nu_g -\int_{0<u_\la<\la}(h+\la^2)(u_\la^+)^{p-1}v^-d\nu_g \] =0
\end{align*}
  Poincar\'e's inequality yields
$$
\int_{\Omega_\la} | \nabla_g v^-|_g^2 d\nu_g\ge C(\Omega_\la)\int_{\Omega_\la}(v^-)^2d\nu_g
$$
where $C(\Omega_\lambda)$ is a positive constant approaching $+\infty$ as $|\Omega_\la|$ goes to zero.

On the other hand
$$
\left|\int_{0< u_\la < \la}(\la^2+h)(u_\la^+)^{p-1}(v^-)^2d\nu_g\right|\leq C\la^{p-1}\|v^-\|^2_{L^2(M)},
$$
and
$$
\left|\int_{0<u_\la<\la}(h+\la^2)(u_\la^+)^{p-1}v^-d\nu_g\right|\leq C\lambda^{p-1}\int_{\Omega_\la} |v^-|d\nu,
$$
for some positive constant $C$   not depending on $\la.$
Collecting the previous computations we get
$$
\underbrace{\(C(\Omega_\la)+c(n)R_g-C\la^{p-1}\)}_{>0\ \hbox{if}\ \lambda\approx 0}\|v^-\|^2_{L^2(M)}+\la\underbrace{\left(c(n)|R_g|-C\la^{p-1} \right)}_{>0\ \hbox{if}\ \lambda\approx 0}\int_{\Omega_\la}|v^-|d\nu_g
\leq 0
$$
which implies $v^-= 0$ if $\la$ is small enough. Since $u_\la\in C^2(M)$ we deduce that
$ 
u_\la \geq \la>0$ in $M
$ and the claim is  proved.
\end{proof}

\subsection{Proof of the main result}
Theorem \ref{main0}   is an immediate consequence of the more general result. 
\begin{theorem}%\label{main}
Assume \eqref{h} with $\gamma\coloneqq 2+\alpha.$ If one of the following conditions hold:
\begin{enumerate}[leftmargin=*]
\item $n\ge10$, the Weyl's tensor at $\xi$ does not vanish, and $2<\alpha<\frac{2n}{n-2}$;
\item $3\le n\le 5$ and $0\le\alpha<n-2$;
\item  $6\le n\le9$ and $\frac{n-6}2 <\alpha<\min\left\{\frac{16}{n-2},\frac{n^2-6n+16}{2(n-2)}\right\}$;
\item  $n\ge10$, $(M,g)$ is locally conformally flat, and $\frac{n-6}{2}<\alpha<\frac{n^2-6n+16}{2(n-2)}$;
\end{enumerate}
then, provided $\lambda$ is small enough, there exists a solution to \eqref{scp} which blows-up at the point $\xi$ as $\lambda\to0.$ 

Moreover, if $h\in C^\infty(M)$ then $\lambda^2+h$ is the scalar curvature of a metric conformal to $g.$
\end{theorem}
\begin{proof}
We will show that the functions $\Theta_1$ and $\Theta_2$, defined respectively in \eqref{rido1} and \eqref{rido2}, have a non-degenerate critical point provided $\sum\limits_{i=1}^n a_i>0$ and $a_i\not=0$ for any $i.$ As a consequence, provided $\la$ is small enough, the reduced energy $\mathcal J_\la$ has a critical point and by Lemma \ref{positivo} we deduce the existence of a classical solution to problem \eqref{scp}, which concludes the proof.

Without loss of generality, we can  consider the function
$$\Theta(t,\tau)\coloneqq t^\beta-t^\gamma\sum\limits_{i=1}^na_i\int\limits_{\mathbb R^n}|y_i+\tau_i|^\gamma f(y)dy,\ (t,\tau)\in(0,+\infty)\times \mathbb R^n,$$ 
where $\beta=4$ or $\beta=\frac{n-2}2$, $\gamma=\alpha+2$ and $f(y)=\frac A{\(1+|y|^2\)^n}$ for some positive constant $A.$
It is immediate to check that, because $\sum\limits_{i=1}^n a_i>0$, this function has a critical point $(t_0,0)$, where $t_0$ solves
$$
\beta t^\beta=\mathfrak c_1 \gamma t^\gamma \sum\limits_{i=1}^na_i,\quad \hbox{with}\ \mathfrak c_1\coloneqq \int\limits_{\mathbb R^n}|y_i |^\gamma f(y)dy\ \hbox{not depending on $i$},
$$
Moreover it is non-degenerate. Indeed,
a straightforward computation shows that 
 $$D^2\phi(t_0,0)=\(\begin{matrix}
 &\beta(\beta-\gamma)t_0^{\beta-2}&0&\dots&0\\
 &0&-\gamma(\gamma-1)\mathfrak c_2 t^\gamma_0a_1&\dots&0\\
&0&0&\dots&-\gamma(\gamma-1)\mathfrak c_2 t^\gamma_0a_n\\
 \end{matrix}\),$$
 where
 $\mathfrak c_2\coloneqq \int\limits_{\mathbb R^n}|y_i |^{\gamma -2} f(y)dy\
$ does not depend on $i,$ which is 
invertible because $\beta\not=\gamma$, $\beta>0$, and $a_i\not=0$ for any $i.$
\end{proof}

\section{The positive case: proof of Theorem \ref{main1}}\label{sec3}

In this section we find a solution to equation \eqref{scp} in the positive case, i.e. $R_g>0$ only assuming the local behavior \eqref{h} of the function $h$ around the local maximum point $\xi$. 
We build solutions to problem \eqref{scp} which blow-up at $\xi$ as $\la$ goes to zero, by   combining the ideas developed by Esposito, Pistoia and V\'etois \cite{espive1}, the Ljapunov-Schmidt argument used in Section \ref{sec2} and the estimates computed in the Appendix.
We omit all the details of the proof because they can be found (up to minor modifications) in \cite{espive1} and in the Appendix.
We only write the profile of the solutions we are looking for and the reduced energy whose critical points produce solutions to our problem.

\subsection{The ansatz}
Let us recall the construction of the main order term of the solution performed in \cite{espive1}.
 In case $\(M,g\)$ is locally conformally flat, there exists a family $\(g_\xi\)_{\xi\in M}$ of smooth conformal metrics to $g$ such that $g_\xi$ is flat in the geodesic ball $B_\xi\(r_0\)$. In case $\(M,g\)$ is not locally conformally flat, we fix $N>n$, and we find a family $\(g_\xi\)_{\xi\in M}$ of smooth conformal metrics to $g$ such that
\begin{equation*} %\label{Eq6}
\left|\exp_\xi^*g_\xi\right|\(y\)=1+O\big(\left|y\right|^N\big)
\end{equation*}
$C^1-$uniformly with respect to $\xi\in M$ and $y\in T_\xi M$, $\left|y\right|\ll1$, where $\left|\exp_\xi^*g_\xi\right|$ is the determinant of $g_\xi$ in geodesic normal coordinates of $g_\xi$ around $\xi$. Such coordinates are said to be {\it conformal normal coordinates} of order $N$ on the manifold. Here, the exponential map $\exp_\xi^*$ is intended with respect to the metric $g_\xi$.  For any $\xi\in M$, we let $\varLambda_\xi$ be the smooth positive function on $M$ such that $g_\xi=\varLambda_\xi^{\frac{4}{n-2}}g$. In both cases (locally conformally flat or not), the metric $g_\xi$ can be chosen smooth with respect to $\xi$ and such that $
 \varLambda_\xi\(\xi\)=1$ and $\nabla\varLambda_\xi\(\xi\)=0$. We let $G_g$ and $G_{g_\xi}$ be the respective Green's functions of $L_g$ and $L_{g_\xi}$. Using the fact that $\varLambda_\xi\(\xi\)=1$,   we deduce 
\begin{equation*} %\label{Eq7}
G_g\(\cdot,\xi\)=\varLambda_\xi (\cdot)\ G_{g_\xi}\(\cdot,\xi\).
\end{equation*}

\medskip
We define 
\begin{equation*} %\label{Eq9}
\mathcal W_{t,\tau}\(x\)=G_g\(x,\xi\)\widehat{\mathcal W}_{t,\tau}\(x\),
\end{equation*}
with
\begin{equation*} %\label{Eq10}
\widehat{\mathcal W}_{t,\tau}\(x\)\coloneqq \left\{\begin{aligned}
&\beta_n\lambda^{-\frac{(n-2)}2}\mu^{-\frac{n-2}{2}}d_{g_\xi}\(x,\xi\)^{n-2}U\(\frac{ d_{g_\xi}(x,\xi)}{\mu}-\tau \)&&\text{if }d_{g_\xi}\(x,\xi\)\le r\\
&\beta_n\lambda^{-\frac{(n-2)}{2}}\mu^{-\frac{n-2}{2}}r^{n-2}U\(\frac{r_0}{\mu}-\tau\)&&\text{if }d_{g_\xi}\(x,\xi\)>r,
\end{aligned}\right.
\end{equation*}
where $\beta_n=(n-2)\omega_{n-1}$, $\omega_{n-1}$ is the volume of the unit $\(n-1\)$--sphere, $\tau\in\mathbb R^n$ and the concentration parameter $\mu=\mu_\la(t)$ with $t>0$ is defined as (here $\alpha=\gamma-2$, being $\gamma$ the order of flatness of $h$ at the point $\xi$)
\begin{equation}\label{Eq8}
\mu=\left\{\begin{aligned}
&t\la^{\frac{2}{4+\alpha-n}}&&\text{if $n=3,4,5$ or  [$n\ge6$ and $\(M,g\)$ is lc.f.] with $n-4<\alpha<n-2$ }\\
&t\ell^{-1}\(\lambda^2\)&&\text{if }n=6\text{ and $\textnormal{Weyl}_g(\xi)\not=0$ with $2<\alpha<4$}\\
&t\la^\frac 2{\alpha-2}&&\text{if }n\ge7\text{ and $\textnormal{Weyl}_g(\xi)\not=0$ with $2<\alpha<n-2$},
\end{aligned}\right.
\end{equation}
where the function $\ell\(\mu\)\coloneqq-\mu^{2-\alpha}\ln\mu$ when $\mu$ is small.
We look for a solution to \eqref{scp} as
$u_\la=\mathcal W_{t,\tau}+\phi_\la$, where the higher order term is found arguing as in Section \ref{sec2}.

\subsection{The reduced energy}
Combining Lemma 1 in \cite{espive1} and Lemma \ref{cruciale} in the Appendix, the reduced energy $\mathcal{J}_\la$  introduced in \eqref{reduced-energy} (where the term $\mathcal U_\la$ is replaced by $\mathcal W_{t,\tau}$ and in particular $u_0=0$) reads as
\begin{enumerate}[leftmargin=*]
\item[(a)] if $n=3,4,5$ or  [$n\ge6$ and $\(M,g\)$ is l.c.f.] with $n-4<\alpha<n-2$ and \eqref{Eq8} holds then   
\begin{equation*} %\label{rido2}
\mathcal{J}_\la(t,\tau)=A_0- \la^\frac{(n-6-\alpha)(n-2)}{4+\alpha-n }\left(  A_3 \mathfrak m(\xi)  t^ {n-2} -A_2t^{2+\alpha}\sum\limits_{i=1}^n\int\limits_{\rr^n}a_i \frac{|y_i+\tau_i|^{2+\alpha}}{(1+|y|^2)^n}dy +o(1)\right),
\end{equation*}
where $\mathfrak m(\xi)>0$ is the {\em mass} at the point $\xi$,
\item[(b)] if $n=6$, $\textnormal{Weyl}_g(\xi)\not=0$, and \eqref{Eq8} holds then  
\begin{equation*} %\label{rido1}
\mathcal{J}_\la(t,\tau)=A_0- \left( -A_1|\textnormal{Weyl}_g(\xi)|_g^2\frac{\mu_\lambda^4(t) \ln\mu_\la(t)}{\la^4}-A_2\frac{\mu_\la^{2+\alpha}(t)}{\la^6}\sum\limits_{i=1}^6\int\limits_{\rr^n}a_i \frac{|y_i+\tau_i|^{2+\alpha}}{(1+|y|^2)^6}dy +o(1)\right),
\end{equation*}
\item[(c)] if $n\ge7$, $\textnormal{Weyl}_g(\xi)\not=0$ and \eqref{Eq8} holds then  
\begin{equation*} %\label{rido1}
\mathcal{J}_\la(t,\tau)=A_0- \la^\frac{2(n+2)-\alpha(n-2)}{\alpha-2}\left( A_1|\textnormal{Weyl}_g(\xi)|_g^2t^4-A_2t^{2+\alpha}\sum\limits_{i=1}^n\int\limits_{\rr^n}a_i \frac{|y_i+\tau_i|^{2+\alpha}}{(1+|y|^2)^n}dy +o(1)\right),
\end{equation*} 
\end{enumerate}
$C^1-$uniformly with respect to $t$ in compact sets of $(0,+\infty)$ and $\tau$ in compact sets of $\mathbb R^n.$
Here $A_1,$ $A_2$ and $A_3$ are constants only depending on $n$ and $A_0$ depends only on  $n$ and $\lambda.$

\section{Appendix} We recall the following useful lemma (see, for example, \cite{yyl}).
\begin{lemma}\label{yyl}
For any $a>0$ and $b\in\mathbb R$ we have
$$
\left|\left((a+b)\right)^q- a^q\right|\le\left\{ 
\begin{array}{ll}
c(q)\min\left\{|b|^q,a^{q-1}|b| \right\} & \hbox{if}\ 0<q<1,\\
c(q)\left(|b|^q+a^{q-1}|b|\right)& \hbox{if}\ q\ge1,
\end{array}
\right.
$$
and
$$
\left|\left((a+b)^+\right)^{q+1}-a^{q+1}-(q+1)a^q b\right|\le\left\{ 
\begin{array}{ll}
c(q)\min\left\{|b|^{q+1},a^{q-1}b^2\right\}& \hbox{if}\ 0<q<1,\\
c(q)\left(|b|^{q+1}+a^{q-1}b^2\right)& \hbox{if}\ q\ge1.
\end{array}
\right.
$$
\end{lemma}

\subsection{Estimate of the error}\label{errore-proof}
\begin{proof}[Proof of Lemma \ref{errore}]
We split the error \eqref{error} into
$$
E=(-\Delta_g+c(n)R_g)\left[\la^{-\frac{n-2}2}\W_{t,\tau}+u_0\right] -(\la^2+h)\left[\la^{-\frac{n-2}2}\W_{t,\tau}+u_0\right]^p=E_1+E_2+E_3, 
$$
where
\begin{align*}
E_1&=\la^{-\frac{n-2}2}\left[-\Delta_g \W_{t,\tau}+c(n)R_g\W_{t,\tau} +\W_{t,\tau}^p\right],\\
E_2&=-\la^{-\frac{n-2}2}\left[(\W_{t,\tau}+\la^{\frac{n-2}2}u_0)^p-\W_{t,\tau}^p \right],\\
E_3&=\la^{-\frac{n+2}2}h\left[(\W_{t,\tau}+\la^{\frac{n-2}2}u_0)^p-(\la^{\frac{n-2}2}u_0)^p \right].
\end{align*}
To estimate $E_2$ and $E_3$ we use the fact that the bubble $\W_{t,\tau}$ satisfies in the three cases
\begin{equation}\label{sti-crux}
\W_{t,\tau}(\textrm{exp}_\xi(y))=O\(\frac{\mu^\frac{n-2}{ 2}}{\left(\mu^2+|y-\mu\tau|^2\)^\frac{n-2}{2}}\right)\quad\hbox{if}\ |y-\xi|\le r.
\end{equation}
Indeed by \eqref{sti-crux} and Lemma \ref{yyl} we immediately deduce that
\begin{align*}
\|E_2\|_{L^{\frac{2n}{n+2}}(M)}&=O\left(\la^{-\frac{n-2}2}\left\|\la^{\frac{n-2}2}u_0\W_{t,\tau}^{p-1}\right\|_{L^{\frac{2n}{n+2}}(M)}\right)+O\left({\la^{-\frac{n-2}2}\left\|\(\la^{\frac{n-2}2}u_0\)^{p}\right\|_{L^{\frac{2n}{n+2}}(M)}}\right)\\ &=O\left(\left\|\W_{t,\tau}^{p-1}\right\|_{L^{\frac{2n}{n+2}}(M)} \right)+O\(\lambda^2\),
\end{align*}
with
$$
\left\|\W_{t,\tau}^{p-1}\right\|_{L^{\frac{2n}{n+2}}(M)}=\left\lbrace
\begin{array}{ll}
\displaystyle O\left(\mu^{\frac{n-2}{2}}\right)&\mbox{if }3\le n\le5,\\
\displaystyle O\left(\mu^2|\ln\mu|^{\frac23}\right)&\mbox{if }n=6,\\
\displaystyle O\left(\mu^2\right)&\mbox{if }n\ge7,
\end{array}\right.
$$
and
\begin{align*}
\|E_3\|_{L^{\frac{2n}{n+2}}(M)}&=
O\left({\la^{-\frac{n+2}2}}\left\|h\(\la^{\frac{n-2}2}u_0\)^{p-1}\W_{t,\tau}\right\|_{L^{\frac{2n}{n+2}}(M)} \right)
+	O\left(\la^{-\frac{n+2}2}\left\|h{\W_{t,\tau}^p}\right\|_{L^{\frac{2n}{n+2}}(M)} \right)\\ &=O\left({\la^{-\frac{n-2}2}}\left\|h\W_{t,\tau}\right\|_{L^{\frac{2n}{n+2}}(M)} \right)
+	O\left(\la^{-\frac{n+2}2}\left\|h{\W_{t,\tau}^p}\right\|_{L^{\frac{2n}{n+2}}(M)} \right),
\end{align*}
with
$$
\left\|h  \W_{t,\tau} \right\|_{L^{\frac{2n}{n+2}}(M)}=\left\lbrace
\begin{array}{ll}
\displaystyle O\left( \mu^{\frac{n-2}{2}}\right)&\mbox{if }n<10+2\alpha,\\
\displaystyle O\left( \mu^{\frac{n-2}{2}} |\ln\mu|^{\frac{n+2}{2n}} \right)&\mbox{if }n=10+2\alpha,\\
\displaystyle O\left( \mu^{4+\alpha}  \right)&\mbox{if }n>10+2\alpha,
\end{array}\right.
$$
and
$$
\left\|h \W_{t,\tau}^p \right\|_{L^{\frac{2n}{n+2}}(M)}=\left\lbrace
\begin{array}{ll}
\displaystyle O\left( \mu^{2+\alpha}\right)&\mbox{if }2\alpha<n-2,\\
\displaystyle O\left( \mu^\frac{n+2}{2}  |\ln\mu|^\frac{n+2}{2n}\right)&\mbox{if } 2\alpha=n-2,\\
\displaystyle O\left( \mu^\frac{n+2}{2}  \right)&\mbox{if } 2\alpha>n-2.
\end{array}\right.
$$
Now, let us estimate $E_1$.
In the first two cases, we argue exactly as in Lemma 3.1 in \cite{espi} and we deduce that
$$
\|E_1\|_{L^{\frac{2n}{n+2}}(M)}=\left\lbrace
\begin{array}{ll}
\displaystyle 
O\left(\frac{\mu^{\frac{n-2}2}}{\la^{\frac{n-2}2}}\right)&\mbox{if } 3\le n\le 7,\\
\displaystyle O\left(\frac{\mu^3|\ln \mu|^{\frac58}}{\la^{3}}\right)&\mbox{if }n=8,\\
\displaystyle O\left(\frac{\mu^3}{\la^{\frac{7}2}}\right)&\mbox{if } n=9,\\
\displaystyle O\left(\frac{\mu^{2\frac{n+2}{ n-2}}}{\la^{\frac{n-2}2}}\right)&\mbox{if } n\geq 10,\\
\end{array}\right.
$$
In the third case, arguing exactly as in Lemma 7.1 of \cite{rove} we get
$$
\|E_1\|_{L^{\frac{2n}{n+2}}(M)} = O\left(\frac{\mu^{\frac{n-2}2}}{\la^{\frac{n-2}2}}\right).
$$
Collecting all the previous estimates we get the claim.
\end{proof}

\subsection{The linear theory}\label{lin-proof}
\begin{proof}[Proof of Lemma \ref{PropLinearTheory}]

We prove \eqref{LinTheEstimate} by contradiction. If the statement were false, there would exist sequences $(\lambda_m)_{m\in \mathbb N},$ $ (t_m)_{m\in \mathbb N}$, $(\tau_m)_{m\in \mathbb{N}}$ such that (up to subsequence) $\lambda_m \downarrow 0$, $\frac{\mu_m}{\la_m}\downarrow 0$, $t_m\to t_0>0$ and $\tau_m\to \tau_0\in\rr^n$ and functions $\phi_m$, $\ell_m$ with $\|\phi_m\|_{H_g^1(M)}=1$, 
$\|\ell_m\|_{L^{\frac{2n}{n+2}}}\to 0$, such that for scalars $c_i^m$ one has
\beq\label{LinearTheory1}
\left \lbrace
\begin{array}{rcll}
L(\phi_m)&=&\displaystyle \ell_m+\sum_{i=0}^nc_i^mZ_{i,t_m,\tau_m}&\mbox{on }M,\\
\displaystyle \int_M \phi_m Z_{i,t_m,\tau_m}d\nu_g&=& 0, &\mbox{for all }i=0,\dots,n.
\end{array}\right.
\eeq
We change variable setting $y=\frac{\exp_{\xi_m }^{-1}(x)}{\mu_m}-\tau_m.$ We remark that $d_g(x,\xi)=|\exp_{\xi }^{-1}(x)|$ and we set
$$
\tilde \phi_m(y)=\mu_m^{\frac{n-2}2}\chi(\mu_m|y+\tau_m|)\phi_m\left(
\exp_{\xi_m}(\mu_m (y+\tau_m))\right)\quad y\in \rr^n.$$
Since $\|\phi_m\|_{H_g^1(M)}=1$, we deduce   that the scaled function $(\tilde \phi_m)_m$ is bounded in $D^{1,2}(\rr^n)$. Up to subsequence, $\tilde \phi_m$ converges weakly to a function $\tilde \phi\in D^{1,2}(\rr^n)$  and thus in $L^{p+1}(\rr^n)$ due to the continuity of the embedding of $D^{1,2}(\rr^n)$ into $L^{p+1}(\rr^n)$.

\medskip
{\textbf{Step 1:}} We show that $c_i^m\to 0$ as $m\to\infty$ for all $i=0,\dots,n$.

\medskip
We test \eqref{LinearTheory1} against $Z_{i,t_m,\tau_m}$. Integration by parts gives
\begin{equation}\label{TestZi}
\begin{split}
\int_M \langle \nabla_g \phi_m,\nabla_g Z_{i,t_m,\tau_m}\rangle_gd\nu_g &+\int_M \left[R_g-(\lambda_m^2+h)f'(\UU_{\la_m})\right] \phi_m Z_{i,t_m,\tau_m}d\nu_g\\
&=\int_M \ell_m Z_{i,t_m,\tau_m}d\nu_g+\sum_{j=0}^n c_i^m\int_M Z_{j,t_m,\tau_m}Z_{i,t_m,\tau_m}d\nu_g.
\end{split}
\end{equation}
Observe that 
$$
\left|\int_M \ell_m Z_{i,t_m,\tau_m}d\nu_g\right|\leq \|\ell_m\|_{L^{\frac{2n}{n+2}}(M)}\|Z_{i,t_m,\tau_m}\|_{L^{\frac{2n}{n-2}}(M)}=o(1).
$$ 
By change of variables we have
$$
c_j^m\int_M Z_{j,t_m,\tau_m}Z_{i,t_m,\tau_m}d\nu_g=c_j^m\int_{\rr^n}Z_jZ_idy+o(1)=c_j^m\delta_{ij}\int_{\rr^n}Z_j^2dy+o(1),
$$
where $\delta_{ij}=1$ if $i=j$ and $0$ otherwise.

Writing $\tilde h(y)=h\left(\exp_{\xi_m}(\mu_m (y+\tau_m))\right)$, note also that
\begin{equation*}
\begin{split}
R_g\int_M \phi_m Z_{i,t_m,\tau_m}d\nu_g&+\int_Mhf'(\UU_{\lambda_m})\phi_m Z_{i,t_m,\tau_m}d\nu_g\\&= R_g\mu_m^2\int_{\rr^n}\tilde \phi_mZ_idy+\int_{\rr^n}\frac{\tilde h}{\lambda_m^2} \tilde \phi_m f'(U)Z_idy+o(1).
\end{split}
\end{equation*}
On the other hand, standard computations show that
\begin{align*}
\int_M \langle \nabla_g \phi_m,\nabla_g Z_{i,t_m,\tau_m}\rangle_g d\nu_g & -\la_m^2\int_M f'(\UU_{\lambda_m})\phi_m Z_{i,t_m,\tau_m}d\nu_g\\
&=\int_{\rr^n}\nabla \tilde \phi_m \cdot \nabla Z_idy-\int_{\rr^n}f'(U)\tilde \phi_m Z_idy+o(1)\\
&=-\int_{\rr^n}(\Delta Z_i+f'(U)Z_i)\tilde \phi_mdy+o(1).
\end{align*}
Since $Z_i$ satisfies $-\Delta Z_i=f'(U)Z_i$ in $\rr^n$, passing to the limit into \eqref{TestZi} yields
$$
\sum_{j=0}^n\lim_{m\to \infty}c_i^m \delta_{ij}\int_{\rr^n}Z_j^2dy=o(1).
$$
Hence $\lim_{m\to \infty} c_i^m=0$, for all $i=0,\dots, n$.

\medskip
{\textbf{Step 2:}} We show that $\tilde \phi\equiv 0$.

\medskip
Given any smooth function $\tilde \psi$ with compact support in $\rr^n$ we define $\psi$ by the relation
$$
\psi (x)=\mu_m^{-\frac{n-2}2}\chi(d_g(x,\xi))\tilde\psi\left(\frac{\exp^{-1}_{\xi_m}(x)}{\mu_m}-\tau_m\right)\quad x\in M.
$$
We test \eqref{LinearTheory1} against $\psi$. Integration by parts gives
\begin{equation}\label{Testpsi}
\begin{split}
\int_M \langle \nabla_g \phi_m,\nabla_g \psi\rangle_gd\nu_g &+\int_M \left[R_g-(\lambda_m^2+h)f'(\UU_{\la_m})\right] \phi_m \psi d\nu_g\\&=
\int_M \ell_m \psi d\nu_g+\sum_{j=0}^n c_i^m\int_M Z_{j,t_m,\tau_m}\psi d\nu_g.
\end{split}
\end{equation}
By {\textbf{Step 1}} it is easy to see that
\begin{equation}\label{RHSpsi}
\int_M \ell_m \psi d\nu_g+\sum_{j=0}^n c_i^m\int_M Z_{j,t_m,\tau_m}\psi d\nu_g\to 0,\quad \mbox{as }m\to \infty.
\end{equation}
On the other hand, by the same arguments given in the proof of {\bf Step 1}, we have 
\begin{equation*}
\begin{split}
\int_M \langle \nabla_g \phi_m,\nabla_g \psi\rangle_g d\nu_g+\int_M & \left[R_g-(\lambda_m^2+h) f'(\UU_{\la_m})\right] \phi_m \psi d\nu_g \\& \to \int_{\rr^n}\nabla \tilde \phi \cdot \nabla \tilde \psi dy -\int_{\rr^n}f'(U)\tilde \phi\tilde \psi dy
\end{split}
\end{equation*}
as $m\to \infty$. Hence, passing to the limit into \eqref{Testpsi} and integrating by parts we get
$$
-\int_{\rr^n}(\Delta \tilde \phi+f'(U)\tilde \phi)\tilde \psi dy=0,\quad \mbox{for all } \tilde \psi \in C_c^\infty(\rr^n).
$$
We conclude that $\tilde \phi$ is a solution in $D^{1,2}(\rr^n)$ to $-\Delta v=f'(U)v$ in $\rr^n$. Thus $\tilde \phi=\sum_{j=0}^n \alpha_j Z_j$, for certain scalars $\alpha_j$. But
$$
0=\int_M \phi_m Z_{i,t_m,\tau_m}d\nu_g=\int_{\rr^n} \tilde \phi_m Z_id\nu_g\quad \mbox{for all }i=0,\dots,n.
$$
Passing to the limit we get $\int_{\rr^n}\tilde \phi Z_id\nu_g=0$ for $i=0,\dots,n$, which implies $\alpha_i=0$ for all $i$. 

\medskip
{\textbf{Step 3:}} We show that, up to subsequence, $\phi_m\rightharpoonup 0$ in $H_g^1(M)$.

\medskip
Since $(\phi_m)_m$ is bounded in $H_g^1(M)$, up to subsequence, $\phi_m$ converges weakly to a function $\phi\in H_g^1(M)$, and thus in $L^{p+1}(M)$ due to the continuity of the embedding of $H_g^1(M)$ into $L^{p+1}(M)$. Moreover, $\phi_m\to \phi$ strongly in $L^2(M)$. 

We test equation \eqref{LinearTheory} against a function $\psi\in H_g^1(M)$. Integration by parts gives \eqref{Testpsi}.
Once again, by {\textbf{Step 1}} it is easy to see that \eqref{RHSpsi} holds. By weak convergence
\begin{equation*}
\begin{split}
\int_M \langle \nabla_g \phi_m,\nabla_g \psi\rangle_gd\nu_g &+ R_g\int_M \phi_m\psi d\nu_g \\&\to \int_M \langle \nabla_g \phi,\nabla_g \psi\rangle_gd\nu_g+R_g\int_M \phi \psi d\nu_g,\quad  \mbox{as } m\to \infty.
\end{split}
\end{equation*} 

{\textbf{Claim:}} $\displaystyle -\int_M (\lambda_m^2+h)f'(\UU_{\la_m})\phi_m\psi d\nu_g \to \int_M hf'(u_0)\phi \psi d\nu_g$ as $m\to \infty$.

Assuming the claim is true, passing to the limit into \eqref{Testpsi} gives
$$
\int_M \langle \nabla_g \phi,\nabla_g \psi\rangle_gd\nu_g+R_g\int_M \phi \psi d\nu_g+\int_M hf'(u_0)\phi \psi d\nu_g=0.
$$
Elliptic estimates show that $\phi$ is a classical solution to $-\Delta_g \phi+R_g\phi =+h f'(u_0)\phi$ on $M$. Lemma \ref{nonde} yields $\phi\equiv 0$. 

\medskip
{\textbf{Proof of the claim:}} Note that 
$$
\lambda_m^2\int_M f'(\UU_{\la_m})\phi_m\psi d\nu_g=\lambda_m^2\int_M [f'(\UU_{\la_m})-f'(u_0)]\phi_m\psi d\nu_g + \lambda_m^2\int_M f'(u_0)\phi_m\psi d\nu_g.
$$
We have 
$$
\left|\lambda_m^2 \int_M f'(u_0)\phi_m\psi d\nu_g\right|\newline\leq \lambda_m^ 2\|f'(u_0)\|_{L^\infty(M)}\|\phi_m\|_{L^2(M)}\|\psi\|_{L^2(M)}\to 0,\quad \mbox{as }m\to \infty.
$$
We define
$$
\tilde \psi (y)=\mu_m^{\frac{n-2}2}\chi(\mu_m |y+\tau_m|)\psi(\exp_{\xi_m}(\mu_m(y+\tau_m)))\quad y\in \rr^n.
$$
By \eqref{firstaprox} and change of variables we have
\begin{equation*}
\begin{split}
\Bigg|\lambda_m^2 \int_M [f'(\UU_{\la_m})&-f'(u_0)]\phi_m\psi d\nu_g\Bigg| \\ &\leq C \int_{\rr^n} \frac1{(1+|y|^2)^2}|\tilde \phi_m(y)||\tilde\psi(y)|dy\\
&\leq C \left\|\frac1{(1+|\cdot|^2)^2}\right\|_{L^\frac{n }{2-\epsilon\frac{(n-2)}2}(\rr^n)}\|\tilde \phi_m\|_{L^{\frac{2n}{n-2}\frac1{1+\epsilon}}(\rr^n)}\|\tilde \psi\|_{L^{\frac{2n}{n-2}}(\rr^n)},
%\left(\int_{\rr^n}|\tilde \phi_m|^{\frac{2n}{n-2}} |\tilde \phi_m|^{\frac{2n}{n-2}}   \right)^{\frac{n-2}n}
\end{split}
\end{equation*}
for $0<\epsilon \ll 1$. 
Note that $\|\tilde \psi\|_{L^{\frac{2n}{n-2}}(\rr^n)}\leq C\|\psi\|_{L^{\frac{2n}{n-2}}(M)}$. By {\textbf{Step 2}} $\|\tilde \phi_m\|_{L^{\frac{2n}{n-2}\frac1{1+\epsilon}}(\rr^n)}\to 0$ as $m\to \infty$, since $\frac{2n}{n-2}\frac1{1+\epsilon}<\frac{2n}{n-2}$ . Thus 
$$
\lambda_m^2\int_M f'(\UU_{\la_m})\phi_m\psi d\nu_g\to 0,\quad \mbox{as }m\to \infty.
$$
On the other hand
$$
\int_M h f'(\UU_{\la_m})\phi_m\psi d\nu_g=\int_M h[f'(\UU_{\la_m})-f'(u_0)]\phi_m\psi d\nu_g+ \int_M hf'(u_0)\phi_m\psi d\nu_g.
$$
Dominated convergence theorem yields
$$
\int_M hf'(u_0)\phi_m\psi d\nu_g\to \int_M hf'(u_0)\phi\psi d\nu_g,\quad \mbox{as }m\to \infty.
$$
By \eqref{firstaprox} and change of variables we have
%$$
%\left|\int_M h[f'(\UU_\la)-f'(u_0)]\phi_m\psi \right|\leq Cu_0(0)^{p-2} \int_{\rr^n} \frac{\mu_m^2}{(\mu_m\la_m)^{\frac{n-2}{2}}}\frac{\tilde h(y)}{(1+|y|^2)^{\frac{n-2}{2}}}|\tilde \phi_m(y)||\tilde\psi(y)|dy
%$$
$$
\left|\int_M h[f'(\UU_{\la_m})-f'(u_0)]\phi_m\psi d\nu_g\right|\leq C \frac{\mu_m^2}{\la_m^2} \int_{\rr^n} \frac{|y|^2}{(1+|y|^2)^2}|\tilde \phi_m(y)||\tilde\psi(y)|dy
$$
We have
\begin{align*}
\frac{\mu_m^2}{\la_m^2}\int_{\rr^n} \chi(\mu_m|y|)\frac{|y|^2}{(1+|y|^2)^2}|&\tilde \phi_m(y)||\tilde\psi(y)|dy
\\
&\leq C \frac{\mu_m^2}{\la_m^2} \left\|\frac{\chi(\mu_m|\cdot|)}{(1+|\cdot|^2)}\right\|_{L^{\frac n2}(\rr^n)}\|\tilde \phi_m\|_{L^{\frac{2n}{n-2}}(\rr^n)}\|\tilde \psi\|_{L^{\frac{2n}{n-2}}(\rr^n)}\\
&\leq C \frac{\mu_m^2}{\la_m^2} |\ln \mu_m|\|\tilde \phi_m\|_{L^{\frac{2n}{n-2}}(\rr^n)}\|\tilde \psi\|_{L^{\frac{2n}{n-2}}(\rr^n)}
\end{align*}
By \textbf{Step 2} and our choice of $\mu_m$ in terms of $\la_m$ (see \eqref{parameters}) we conclude that
$$
\left|\int_M h[f'(\UU_{\la_m})-f'(u_0)]\phi_m\psi d\nu_g \right|,\quad \mbox{as }m\to \infty.
$$
The claim is thus proved.

\medskip
{\textbf{Step 4:}} We show that $\|\phi_m\|_{H_g^1(M)}\to 0$.

\medskip
We take in \eqref{Testpsi} $\psi=\phi_m$. We get
\begin{equation}\label{Testphi}
\begin{split}
\int_M |\nabla_g \phi_m|_g^2d\nu_g &+\int_M \left[R_g-(\lambda_m^2+h)f'(\UU_{\la_m})\right] \phi_m^2d\nu_g \\&=
\int_M \ell_m \phi_md\nu_g+\sum_{j=0}^n c_i^m\int_M Z_{j,t_m,\tau_m}\phi_md\nu_g.
\end{split}
\end{equation}
By \textbf{Step 1--3}, passing to the limit into \eqref{Testphi} gives
$$
\lim_{m\to\infty} \int_M |\nabla_g \phi_m|_g^2d\nu_g=0.
$$
Since $\phi_m \rightharpoonup 0$ in $H_g^1(M)$, we conclude
$$
\|\phi_m\|_{H_g^1(M)}\to 0,\quad \mbox{as }m\to \infty,
$$
which yields a contradiction with the fact that $\|\phi_m\|_{H_g^1(M)}=1$. This concludes the proof of \eqref{LinTheEstimate}.

\medskip
The existence and uniqueness of $\phi_\la$ solution to Problem \ref{LinearTheory} follows from the Fredholm alternative. 
This finishes the proof of Lemma \ref{IntermediateProp}.
\end{proof}

\subsection{The reduced energy}\label{nonlin-proof}
\begin{proof}[Proof of Lemma \ref{IntermediateProp}]
The result of Proposition \ref{PropLinearTheory} implies that the unique solution $\phi_\la=T_{t,\tau}(\ell)$ of \eqref{LinearTheory} defines a continuous linear map $T_{t,\tau}$ from the space $L^{\frac{2n}{n+2}}(M)$ into $H_g^1(M)$. Moreover, a standard argument shows that the operator $T_{t,\tau}$ is continuously differentiable with respect to $t$ and $\tau$.

In terms of the operator $T_{t,\tau}$, Problem \ref{IntermediateProblem} becomes
$$
\phi_\la=T_{t,\tau}(-E+(\la^2+h)N(\phi_\la))\eqqcolon A(\phi_\la).
$$
We define the space 
$$
H=\left\{\phi \in H_g^1(M)\; \vline \; \int_M \phi Z_{i,t,\tau}d\nu_g= 0,\, \mbox{for all }i=0,\dots,n  \right\}.
$$
For any positive real number $\eta$, let us consider the region
$$
\mathcal{F}_\eta\equiv \left\{\phi\in H \; \vline \; \|\phi\|_{H_g^1(M)}\leq \eta \|E\|_{L^{\frac{2n}{n+2}}(M)} \right\}.
$$
From \eqref{LinTheEstimate}, we get
$$
\|A(\phi_\la)\|_{H_g^1(M)}\leq C \left(\|E\|_{L^{\frac{2n}{n+2}}(M)}+ \|N(\phi_\la)\|_{L^{\frac{2n}{n+2}}(M)} \right).
$$
Observe that
$$
\|N(\phi_\la)\|_{L^{\frac{2n}{n+2}}(M)}\leq C \|\phi_\la\|_{L^{\frac{2n}{n-2}}(M)}^p\leq C \|\phi_\la\|_{H_g^1(M)},
$$
and
$$
\|N(\phi_1)-N(\phi_2)\|_{L^{\frac{2n}{n+2}}(M)} \leq C \eta^{p-1}\|E\|_{L^{\frac{2n}{n-2}}(M)}^{p-1}\|\phi_1-\phi_2\|_{H_g^1(M)},
$$
for $\phi_1,\phi_2 \in \mathcal{F}_\eta$.
By \eqref{estimatePhiError}, we get
$$
\|A(\phi_\la)\|_{H_g^1(M)} \leq  C\|E\|_{L^{\frac{2n}{n-2}}(M)}\left(\eta^p \|E\|_{L^{\frac{2n}{n-2}}(M)}^{p-1}+1\right),
$$
and
$$
\|A(\phi_1)-A(\phi_2)\|_{H_g^1(M)}\leq  C \eta^{p-1}\|E\|_{L^{\frac{2n}{n-2}}(M)}^{p-1}\|\phi_1-\phi_2\|_{H_g^1(M)},
$$
for $\phi_1,\phi_2 \in \mathcal{F}_\eta$.

Since $p-1\in (0,1)$ for $n\geq 3$ and $\|E\|_{L^{\frac{2n}{n-2}}(M)}\to 0$ as $\lambda\to 0$, it follows that if $\eta$ is sufficiently large and $\lambda_0$ is small enough then $A$ is a contraction map from $\mathcal{F}_\eta$ into itself, and therefore a unique fixed point of $A$ exists in this region.

Moreover, since $A$ depends continuously (in the $L^{\frac{2n}{n+2}}$-norm) on $t,\tau$ the fixed point characterization obviously yields so for the map $t,\tau\to\phi$. Moreover, standard computations give that the partial derivatives $\partial_t\phi,\partial_{\tau_i}\phi$, $i=1,\dots,n$ exist and define continuous functions of $t,\tau$. Besides, there exists a constant $C>0$ such that for all $i=1,\dots,n$
\beq\label{EstimateDerivates}
\left\|\partial_t \phi\right\|_{H_g^1(M)}+\left\|\partial_{\tau_i}\phi\right\|_{H_g^1(M)}\leq C \left\|E\right\|_{L^{\frac{2n}{n+2}}(M)}+\left\|\partial_t E\right\|_{L^{\frac{2n}{n+2}}(M)}+\left\|\partial_{\tau_i} E\right\|_{L^{\frac{2n}{n+2}}(M)}.
\eeq
That concludes the proof. \end{proof}

\subsection{The reduced energy}\label{ridotto-proof}
It is quite standard to prove that, as $\la\to0$, $J_\la\(\UU_\la+\phi_\la\)=J_\la\(\UU_\la \)+h.o.t.$
 $C^1$-uniformly on compact sets of $(0,+\infty)\times \mathbb R^n$ (see \cite{espive1}). It only remains to compute $J_\la\(\UU_\la\).$
\begin{lemma}\label{cruciale}
Let $a,b\in \rr_+$ be fixed numbers such that $0<a<b$ and let $K$ be a compact set in $\rr^n.$ There exists a positive number $\lambda_0$  such that for any $ \lambda\in(0,\lambda_0)$
the following expansions hold  $C^1-$uniformly with respect to $t\in[a,b]$ and $\tau\in K$:
\begin{enumerate}[leftmargin=*]
\item[(a)]
if  $n\ge10$, $|\textnormal{Weyl}_g(\xi)|\neq0$, and \eqref{mu1} holds, we have
$$
J_\la(\UU_\la)=A_0- \la^\frac{2(n+2)-\alpha(n-2)}{\alpha-2}\left[A_1|\textnormal{Weyl}_g(\xi)|_g^2t^4-A_2t^{2+\alpha}\sum\limits_{i=1}^n\int\limits_{\rr^n}a_i \frac{|y_i+\tau_i|^{2+\alpha}}{(1+|y|^2)^n}dy+o(1)\right],
$$
\item[(b)]
if one of the following conditions is satisfied:
\begin{itemize}
\item[(i)] $3\le n\le 5$ and \eqref{mu1} holds;
\item[(ii)] $6\le n\le9$ and \eqref{mu9} holds; 
\item[(iii)] $n\ge10,$ $(M,g)$ is locally conformally flat, and \eqref{mu3} holds;
\end{itemize}
then 
$$
J_\la(\UU_\la)=A_0- \la^\frac{(n-2-\alpha)(n-2)}{2\alpha-n+6 }\left[A_3u_0(\xi)t^\frac{n-2}2-A_2t^{2+\alpha}\sum\limits_{i=1}^n\int\limits_{\rr^n}a_i\frac{|y_i+\tau_i|^{2+\alpha}}{(1+|y|^2)^n}dy+o(1)\right].
$$
\end{enumerate}
Here $A_1,$ $A_2$ and $A_3$ are constants only depending on $n$ and $A_0$ is defined in \eqref{a0}.
\end{lemma}
\begin{proof} 
We prove the $C^0-$estimate. The $C^1-$estimate can be carried out in a similar way (see \cite{espive1}). 

Let us first prove  (a) and (ii) and (iii) of (b).
It is useful to recall that
$$
\alpha<\frac{2n}{n-2}<n-2\ \hbox{if}\ n\ge10,\quad \alpha<\frac{n^2-6n+16}{2(n-2)}<n-2\ \hbox{if}\ n\ge 6.
$$
Observe that 
\begin{align*}
J_\la(\UU_\la) &=\underbrace{\frac12 \int_M |\nabla_g u_0|_g^2d\nu_g+\frac12\int_M c(n)R_gu_0^2d\nu_g+\frac1{p+1}\int_M h  u_0 ^{p+1}d\nu_g}_{\hbox{independent on $\mu$ and $\tau$}}\\ &+\underbrace{\la^{-(n-2)}\left[\frac12 \int_M |\nabla_g \W_{t,\tau}|_g^2d\nu_g+\frac12 \int_M c(n)R_g\W_{t,\tau}^2d\nu_g-\frac1{p+1}\int_M\W_{t,\tau}^{p+1}d\nu_g\right]}_{\hbox{leading term in case (a)}}\\
&+\underbrace{\la^{-\frac{n-2}2}\left[\int_M \langle\nabla_g \W_{t,\tau},\nabla_g u_0\rangle d\nu_g+\int_M c(n)R_g \W_{t,\tau}u_0d\nu_g\right]-\la^{-\frac{n-2}2} \int_Mh \W_{t,\tau}u_0^pd\nu_g}_{=0}\\
&-\la^{-(n-2)}\frac1{p+1}\int_M\[\left(\W_{t,\tau}+\la^{\frac{n-2}2}u_0\)^{p+1}-\W_{t,\tau}^{p+1}-\(\la^{\frac{n-2}2}u_0\)^{p+1}\right.\\ &\hskip4truecm\left.- (p+1)\W_{t,\tau}^{p } \(\la^{\frac{n-2}2}u_0\)- (p+1)\W_{t,\tau} \(\la^{\frac{n-2}2}u_0\)^{p }\right]d\nu_g\\
&+\underbrace{\la^{2}\frac1{p+1}\int_Mu_0 ^{p+1}d\nu_g}_{\hbox{independent of $\mu$ and $\tau$}}\\
&-\underbrace{\la^{-(n-2)}\int_M\W_{t,\tau}^{p } \(\la^{\frac{n-2}2}u_0\)d\nu_g}_{\hbox{leading term in case (b)}}\\
&-\la^{-(n-2)}\int_M\W_{t,\tau} \(\la^{\frac{n-2}2}u_0\)^{p }d\nu_g\\
&+\la^{-n}\frac1{p+1}\int_M h\[ \left(\W_{t,\tau}+\la^{\frac{n-2}2}u_0\)^{p+1}- \W_{t,\tau} ^{p+1}-\( \la^{\frac{n-2}2}u_0\)^{p+1}\right.\\ &\hskip4truecm\left.- (p+1)\W_{t,\tau}^p\(\la^{\frac{n-2}2}u_0\) -(p+1)\W_{t,\tau}\(\la^{\frac{n-2}2}u_0\)^{p}\right]d\nu_g\\
&+\underbrace{\la^{-n}\frac1{p+1}\int_M h\W_{t,\tau} ^{p+1}d\nu_g}_{\hbox{leading term in every case}}\\
&+\la^{-n} \int_M h\W_{t,\tau}^p\(\la^{\frac{n-2}2}u_0\)d\nu_g
\end{align*}
Concerning the leading terms, we need to distinguish two cases.
If the manifold is not locally conformally flat, by Lemma 3.1 in \cite{espi} we deduce
\begin{equation}\label{l1}
\begin{split}
\la^{-(n-2)}&\left[\frac12 \int_M |\nabla_g \W_{t,\tau}|_g^2d\nu_g+ \frac12 \int_M c(n)R_g\W_{t,\tau}^2d\nu_g-\frac1{p+1}\W_{t,\tau}^{p+1}d\nu_g\right]\\ &=\left\{\begin{array}{ll}
\la^{-(n-2)}\left[A(n)-B(n)|\mbox{Weyl}_g(\xi)|_g^2\mu^4+o(\mu^4)\right] & \hbox{if}\ n\ge7,\\
\la^{-(n-2)}\left[A(n)-B(n)|\mbox{Weyl}_g(\xi)|_g^2\mu^4|\ln\mu|+o(\mu^4|\ln\mu|)\right]& \hbox{if}\ n=6.\\
\end{array}\right.
\end{split}
\end{equation}
If the manifold is locally conformally flat, by Lemma 5.2 in \cite{rove} we get
\begin{equation*} %\label{l1-lcf}
 \la^{-(n-2)}\left[\frac12 \left(\int_M |\nabla_g \W_{t,\tau}|_g^2d\nu_g+c(n)R_g\W_{t,\tau}^2\right) d\nu_g-\frac1{p+1}\W_{t,\tau}^{p+1}d\nu_g\right] =A(n)+O\(\frac{\mu^{n-2}}{\la^{n-2}}\).
\end{equation*}
Here $A(n)$and $B(n)$ are positive constants  independent only depending on $n.$
 Moreover,
 straightforward computations lead to 
\begin{equation}\label{l2}\la^{-(n-2)}\int_M\W_{t,\tau}^{p } \(\la^{\frac{n-2}2}u_0\)d\nu_g=u_0(\xi)\frac{\mu^\frac{n-2}2}{\lambda^\frac{n-2}2}\alpha_n^p\int _{\rr^n}
\frac1{(1+|y|^2)^\frac{n+2}2}dy+o\(\frac{\mu^\frac{n-2}2}{\lambda^\frac{n-2}2}\)\end{equation}
and
\begin{equation}\label{l3} 
 \la^{-n}\frac1{p+1}\int_M h  \W_{t,\tau} ^{p+1}d\nu_g =
\frac{\mu^{2+\alpha}}{\lambda^n}\frac{\alpha_n^{p+1}}{p+1}\sum\limits_{i=1}^n\int_{\rr^n}a_i \frac{|y_i+\tau_i|^{2+\alpha}}{(1+|y|^2)^n}dy +o\(\frac{\mu^{2+\alpha}}{\lambda^n}\),\end{equation}
 because $\alpha<n-2.$
 
\medskip 
Now, if $n\ge10$ and the manifold is not locally conformally flat, we choose $\mu=t \lambda^\frac2{\alpha-2}$   so that the leading terms are \eqref{l1} and \eqref{l3}, namely
$$
\frac{\mu^4}{\lambda^{n-2}}\sim\frac{\mu^{2+\alpha}}{\lambda^n}\quad \hbox{and}\quad \frac{\mu^\frac{n-2}2}{\lambda^\frac{n-2}2}=
o\left(\frac{\mu^4}{\lambda^{n-2}}
 \).
$$
On the other hand, if $6\le n\le 9$ we choose $\mu=t \lambda^\frac{n+2}{2\alpha-n+6}$ so that the leading terms are \eqref{l2} and \eqref{l3}, namely
$$
\frac{\mu^\frac{n-2}2}{\lambda^\frac{n-2}2}\sim
\frac{\mu^{2+\alpha}}{\lambda^n}\quad \hbox{and}\quad \frac{\mu^4}{\lambda^{n-2}}=
o\(\frac{\mu^\frac{n-2}2}{\lambda^\frac{n-2}2}
\), \quad \hbox{provided that}\ \alpha<\frac{16}{n-2}.
$$
The higher order terms are estimated only taking into account that the bubble $\W_{t,\tau} $ satisfies \eqref{sti-crux}.

A simple computation shows that
$$\la^{-(n-2)}\int_M\W_{t,\tau} \(\la^{\frac{n-2}2}u_0\)^{p }d\nu_g=O\(\frac{\mu^\frac{n-2}2}{\lambda^\frac{n-2}2}\lambda^2\int_{B (0,r)} \frac1{|y-\mu \tau|^{n-2}} dy \)=o\(\frac{\mu^\frac{n-2}2}{\lambda^\frac{n-2}2}\) 
$$
and
$$\la^{-n} \int_M h\W_{t,\tau}^p\(\la^{\frac{n-2}2}u_0\)d\nu_g=O\(\frac{\mu^\frac{n+2}2}{\lambda^\frac{n+2}2} \int_{B (0,r)} \frac1{|y-\mu \tau|^{n-\alpha}} dy \)=o\(\frac{\mu^\frac{n-2}2}{\lambda^\frac{n-2}2}\).$$
If $n=6$ then $p+1=3$. It follows that 
 \begin{align*}
&
 \la^{-(n-2)}\frac1{p+1}\int_M\[\left(\W_{t,\tau}+\la^{\frac{n-2}2}u_0\)^{p+1}-\W_{t,\tau}^{p+1}-\(\la^{\frac{n-2}2}u_0\)^{p+1}\right.\\ &\hskip4truecm\left.- (p+1)\W_{t,\tau}^{p } \(\la^{\frac{n-2}2}u_0\)- (p+1)\W_{t,\tau} \(\la^{\frac{n-2}2}u_0\)^{p }\right]d\nu_g=0
 \end{align*}
and
 \begin{align*}
&\la^{-n}\frac1{p+1}\int_M h\[ \left(\W_{t,\tau}+\la^{\frac{n-2}2}u_0\)^{p+1}- \W_{t,\tau} ^{p+1}-\( \la^{\frac{n-2}2}u_0\)^{p+1}\right.\\ &\hskip4truecm\left.- (p+1)\W_{t,\tau}^p\(\la^{\frac{n-2}2}u_0\) -(p+1)\W_{t,\tau}\(\la^{\frac{n-2}2}u_0\)^{p}\right]d\nu_g=0.
\end{align*}
If $n\ge7$, we get
 \begin{align*}
&
 \la^{-(n-2)}\frac1{p+1}\int_M\[\left(\W_{t,\tau}+\la^{\frac{n-2}2}u_0\)^{p+1}-\W_{t,\tau}^{p+1}-\(\la^{\frac{n-2}2}u_0\)^{p+1}\right.\\ &\hskip4truecm\left.- (p+1)\W_{t,\tau}^{p } \(\la^{\frac{n-2}2}u_0\)- (p+1)\W_{t,\tau} \(\la^{\frac{n-2}2}u_0\)^{p }\right]d\nu_g\\
 &=   \la^{-(n-2)}\frac1{p+1}\int_{B_g(\xi,\sqrt\mu)}\[\left(\W_{t,\tau}+\la^{\frac{n-2}2}u_0\)^{p+1}-\W_{t,\tau}^{p+1}-\(\la^{\frac{n-2}2}u_0\)^{p+1}\right.\\ &\hskip4truecm\left.- (p+1)\W_{t,\tau}^{p } \(\la^{\frac{n-2}2}u_0\)- (p+1)\W_{t,\tau} \(\la^{\frac{n-2}2}u_0\)^{p }\right]d\nu_g\\ &+  \la^{-(n-2)}\frac1{p+1}\int_{M\setminus B_g(\xi,\sqrt\mu)}\[\left(\W_{t,\tau}+\la^{\frac{n-2}2}u_0\)^{p+1}-\W_{t,\tau}^{p+1}-\(\la^{\frac{n-2}2}u_0\)^{p+1}\right.\\ &\hskip4truecm\left.- (p+1)\W_{t,\tau}^{p } \(\la^{\frac{n-2}2}u_0\)- (p+1)\W_{t,\tau} \(\la^{\frac{n-2}2}u_0\)^{p }\right]d\nu_g\\
 &=o\(\frac{\mu^\frac{n-2}2}{\lambda^\frac{n-2}2}\)
,\end{align*}
 because by Lemma \ref{yyl} 
\begin{align*}
&\la^{-(n-2)}\frac1{p+1}\int_{B_g(\xi,\sqrt\mu)}\[\left(\W_{t,\tau}+\la^{\frac{n-2}2}u_0\)^{p+1}-\W_{t,\tau}^{p+1}-\(\la^{\frac{n-2}2}u_0\)^{p+1}\right.\\ &\hskip4truecm\left.- (p+1)\W_{t,\tau}^{p } \(\la^{\frac{n-2}2}u_0\)- (p+1)\W_{t,\tau} \(\la^{\frac{n-2}2}u_0\)^{p }\right]d\nu_g\\ 
&=\la^{-(n-2)}\frac1{p+1}\int_{  B_g(\xi,\sqrt\mu)}\[\left(\W_{t,\tau}+\la^{\frac{n-2}2}u_0\)^{p+1}-\W_{t,\tau}^{p+1}- (p+1)\W_{t,\tau}^{p } \(\la^{\frac{n-2}2}u_0\)\right]d\nu_g\\
&-\la^{-(n-2)}\frac1{p+1}\int_{ B_g(\xi,\sqrt\mu)}\(\la^{\frac{n-2}2}u_0\)^{p+1}d\nu_g-\la^{-(n-2)}\int_{ B_g(\xi,\sqrt\mu)}  \W_{t,\tau} \(\la^{\frac{n-2}2}u_0\)^{p }d\nu_g\\ 
&=O\(\la^{-(n-2)}\int_{  B_g(\xi,\sqrt\mu)}\left(\la^{\frac{n-2}2}u_0\)^{p+1} d\nu_g\right)+O\(\la^{-(n-2)}\int_{  B_g(\xi,\sqrt\mu)}\W_{t,\tau}^{p-1}\left(\la^{\frac{n-2}2}u_0\)^{2} d\nu_g\right)\\
&+O\(\la^{-(n-2)}\int_{ B_g(\xi,\sqrt\mu)}  \W_{t,\tau} \left(\la^{\frac{n-2}2}u_0\)^{p }d\nu_g\right)\\
& =O\(\la^2\int_{ B_g(\xi,\sqrt\mu)} u_0 ^{p+1}d\nu_g\)+O\(\mu^2\int_{ B (0,\sqrt\mu)} \frac1{|y-\mu \tau|^4}dy\) \\
&+O\(\frac{\mu^\frac{n-2}2}{\lambda^\frac{n-6}2} \int_{ B (0,\sqrt\mu)} \frac1{|y-\mu \tau|^{n-2}}dy\) \\ 
&=O\(\la^2\mu^\frac{n}2\)+O\(\mu^{2+{\frac{n-4}2}}\)+O\(\frac{\mu^\frac{n-2}2}{\lambda^\frac{n-6}2} \mu \)=o\(\frac{\mu^\frac{n-2}2}{\lambda^\frac{n-2}2}\),
\end{align*}
since if $\mu$ is small enough, for any $q<n$, we have
\beq\label{star}
\int_{ B (0,\sqrt\mu)}  \frac1{|y-\mu \tau|^{q}}dy=\int_{ B (-\mu\tau,\sqrt\mu)}  \frac1{|y|^{q}}dy\le \int_{ B (0,2\sqrt\mu)}  \frac1{|y |^{q}}dy 
=O\(\mu^\frac{n-q}2\)\eeq
and 
\begin{align*}
&
 \la^{-(n-2)}\frac1{p+1}\int_{M\setminus B_g(\xi,\sqrt\mu)}\[\left(\W_{t,\tau}+\la^{\frac{n-2}2}u_0\)^{p+1}-\W_{t,\tau}^{p+1}-\(\la^{\frac{n-2}2}u_0\)^{p+1}\right.\\ &\hskip4truecm\left.- (p+1)\W_{t,\tau}^{p } \(\la^{\frac{n-2}2}u_0\)- (p+1)\W_{t,\tau} \(\la^{\frac{n-2}2}u_0\)^{p }\right]d\nu_g\\ &=
  \frac{\la^{-(n-2)}}{p+1}\int_{M\setminus B_g(\xi,\sqrt\mu)}\[\left(\W_{t,\tau}+\la^{\frac{n-2}2}u_0\)^{p+1}-\(\la^{\frac{n-2}2}u_0\)^{p+1}- (p+1)\W_{t,\tau} \(\la^{\frac{n-2}2}u_0\)^{p }\right]d\nu_g\\ & - \frac{\la^{-(n-2)}}{p+1}\int_{M\setminus B_g(\xi,\sqrt\mu)}\W_{t,\tau}^{p+1}d\nu_g-  \la^{-(n-2)} \int_{M\setminus B_g(\xi,\sqrt\mu)} \W_{t,\tau}^{p } \(\la^{\frac{n-2}2}u_0\)d\nu_g \\ &
 =O\(  \la^{-(n-2)} \int_{M\setminus B_g(\xi,\sqrt\mu)}\W_{t,\tau}^{2}\left(\la^{\frac{n-2}2}u_0\)^{p-1}d\nu_g\right)+O\(\frac{\la^{-(n-2)}}{p+1}\int_{M\setminus B_g(\xi,\sqrt\mu)}\W_{t,\tau}^{p+1}d\nu_g\)\\ &+O\(\la^{-\frac{n-2}2} \int_{M\setminus B_g(\xi,\sqrt\mu)} \W_{t,\tau}^{p } d\nu_g  \)\\ &
 =O\(  \la^{-(n-4)}\mu^\frac{n}2\)+O\(- \la^{-(n-2)}\mu^\frac{n}2\)+O\( \frac{\mu^\frac{n}2}{\la^\frac{n-2}2}   \)=o\(\frac{\mu^\frac{n-2}2}{\lambda^\frac{n-2}2}  \),
\end{align*}
Here, we used the fact that 
$$
\frac{\mu^\frac{n}2}{\lambda^{n-2}}= o\(\frac{\mu^\frac{n-2}2}{\lambda^\frac{n-2}2}\),
$$
and our choice
$$
\begin{array}{rcll}
&\alpha&<\displaystyle \frac{n^2-6n+16}{2(n-2)}& \hbox{if}\ 7\leq n\leq 9,\\
\displaystyle\frac{n-6}2<&\alpha&< \displaystyle\frac{n^2-6n+16}{2(n-2)} & \hbox{if $n\ge 10$ and $(M,g)$ is locally conformally flat},\\
&\alpha&<\displaystyle\frac{2n}{n-2}&\hbox{if $n\ge 10$ and $(M,g)$ is not locally conformally flat}.
\end{array}
$$ 
%  \textcolor{red}{ 
%  $$ \begin{aligned}
%&\la^{-\frac{n-2}2} \int_{M\setminus B_g(\xi,\sqrt\mu)} \W_{t,\tau}^{p }\le c {\mu^{n+2\over 2}\over\la^{\frac{n-2}2}}  \int_{B(0,r)\setminus B(0,\sqrt\mu)}{1\over \(\mu^2+|x+\mu \tau|^2\)^{n+2\over2}}\\ & \hbox{(we scale $x =\mu y $)}\\ &
%\le c {\mu^{n+2\over 2}\over\la^{\frac{n-2}2}}  \mu^{- 2}\int_{\{|y|\ge{1\over\sqrt\mu}\}}{1\over \(1+|y+\tau|^2\)^{n+2\over2}}\\ &
%\le c {\mu^{n+2\over 2}\over\la^{\frac{n-2}2}}  \mu^{- 2}\int_{{1\over\sqrt\mu}}^{+\infty}{\rho^{n-1}\over  \rho^{n+2 }}
%\\ &\le c {\mu^{n+2\over 2}\over\la^{\frac{n-2}2}}  \mu^{- 2}\mu\le c {\mu^{n-2\over 2}\over\la^{\frac{n-2}2}}   \mu
% \end{aligned}$$}

In a similar way, if $n\ge 7$, we get
 \begin{align*}
&
 \la^{-n}\frac1{p+1}\int_Mh\[\left(\W_{t,\tau}+\la^{\frac{n-2}2}u_0\)^{p+1}-\W_{t,\tau}^{p+1}-\(\la^{\frac{n-2}2}u_0\)^{p+1}\right.\\ &\hskip4truecm\left.- (p+1)\W_{t,\tau}^{p } \(\la^{\frac{n-2}2}u_0\)- (p+1)\W_{t,\tau} \(\la^{\frac{n-2}2}u_0\)^{p }\right]d\nu_g\\
 &=   \la^{-n}\frac1{p+1}\int_{B_g(\xi,\sqrt\mu)}h\[\left(\W_{t,\tau}+\la^{\frac{n-2}2}u_0\)^{p+1}-\W_{t,\tau}^{p+1}-\(\la^{\frac{n-2}2}u_0\)^{p+1}\right.\\ &\hskip4truecm\left.- (p+1)\W_{t,\tau}^{p } \(\la^{\frac{n-2}2}u_0\)- (p+1)\W_{t,\tau} \(\la^{\frac{n-2}2}u_0\)^{p }\right]d\nu_g\\ &+  \la^{-n}\frac1{p+1}\int_{M\setminus B_g(\xi,\sqrt\mu)}h\[\left(\W_{t,\tau}+\la^{\frac{n-2}2}u_0\)^{p+1}-\W_{t,\tau}^{p+1}-\(\la^{\frac{n-2}2}u_0\)^{p+1}\right.\\ &\hskip4truecm\left.- (p+1)\W_{t,\tau}^{p } \(\la^{\frac{n-2}2}u_0\)- (p+1)\W_{t,\tau} \(\la^{\frac{n-2}2}u_0\)^{p }\right]d\nu_g\\
 &= {o\(\frac{\mu^{2+\alpha}}{\lambda^n}  \)},
\end{align*}
because by Lemma \ref{yyl}
\begin{align*}
&
  \la^{-n}\frac1{p+1}\int_{B_g(\xi,\sqrt\mu)}h\[\left(\W_{t,\tau}+\la^{\frac{n-2}2}u_0\)^{p+1}-\W_{t,\tau}^{p+1}-\(\la^{\frac{n-2}2}u_0\)^{p+1}\right.\\ &\hskip4truecm\left.- (p+1)\W_{t,\tau}^{p } \(\la^{\frac{n-2}2}u_0\)- (p+1)\W_{t,\tau} \(\la^{\frac{n-2}2}u_0\)^{p }\right]d\nu_g\\ &= \la^{-n}\frac1{p+1}\int_{  B_g(\xi,\sqrt\mu)}h\[\left(\W_{t,\tau}+\la^{\frac{n-2}2}u_0\)^{p+1}-\W_{t,\tau}^{p+1}- (p+1)\W_{t,\tau}^{p } \(\la^{\frac{n-2}2}u_0\)\right]d\nu_g\\
  &-\la^{-n}\frac1{p+1}\int_{ B_g(\xi,\sqrt\mu)}h\(\la^{\frac{n-2}2}u_0\)^{p+1}d\nu_g-\la^{-n}\int_{ B_g(\xi,\sqrt\mu)} h \W_{t,\tau} \(\la^{\frac{n-2}2}u_0\)^{p }d\nu_g\\ & =O\(\la^{-n}\int_{  B_g(\xi,\sqrt\mu)}h\left(\la^{\frac{n-2}2}u_0\)^{p+1}d\nu_g \right)+O\(\la^{-n}\int_{  B_g(\xi,\sqrt\mu)}h\W_{t,\tau}^{p-1}\left(\la^{\frac{n-2}2}u_0\)^{2}d\nu_g \right)
  \\ &+O\(\la^{-n}\int_{ B_g(\xi,\sqrt\mu)}  h\W_{t,\tau} \left(\la^{\frac{n-2}2}u_0\)^{p }d\nu_g\right)\\
  &
 =O\(\int_{ B_g(\xi,\sqrt\mu)}\left(d_g(x,\xi)\)^{\alpha+2}d\nu_g\right)+O\(\frac{\mu^2}{\lambda^2} \int_{ B (0,\sqrt\mu)}  \frac{|y|^{\alpha+2}}{|y-\mu\tau|^{4}}dy \)\\
  &+O\(\frac{\mu^\frac{n-2}2}{\lambda^\frac{n-2}2} \int_{ B (0,\sqrt\mu)}  \frac{|y|^{\alpha+2}}{|y-\mu\tau|^{n-2}} dy \)\\ 
  &=O\( \mu^\frac{\alpha+2+n}2\)+O\(\frac{\mu^\frac{\alpha+2+n}2}{\lambda^2}\)+O\(\frac{\mu^\frac{n-2}2}{\lambda^\frac{n-2}2}\mu ^\frac{\alpha+4}2\)=o\(\frac{\mu^{2+\alpha}}{\lambda^n}  \)
  \end{align*}
Here, we used \eqref{star} and 
\begin{align*}
& 
 \la^{-n}\frac1{p+1}\int_{M\setminus B_g(\xi,\sqrt\mu)}h\[\left(\W_{t,\tau}+\la^{\frac{n-2}2}u_0\)^{p+1}-\W_{t,\tau}^{p+1}-\(\la^{\frac{n-2}2}u_0\)^{p+1}\right.\\ &\hskip4truecm\left.- (p+1)\W_{t,\tau}^{p } \(\la^{\frac{n-2}2}u_0\)- (p+1)\W_{t,\tau} \(\la^{\frac{n-2}2}u_0\)^{p }\right]d\nu_g\\ &=
  \frac{\la^{-n}}{p+1}\int_{M\setminus B_g(\xi,\sqrt\mu)}h\[\left(\W_{t,\tau}+\la^{\frac{n-2}2}u_0\)^{p+1}-\(\la^{\frac{n-2}2}u_0\)^{p+1}- (p+1)\W_{t,\tau} \(\la^{\frac{n-2}2}u_0\)^{p }\right]d\nu_g\\ & - \frac{\la^{-n}}{p+1}\int_{M\setminus B_g(\xi,\sqrt\mu)}h\W_{t,\tau}^{p+1}d\nu_g-  \la^{-n} \int_{M\setminus B_g(\xi,\sqrt\mu)} h\W_{t,\tau}^{p } \(\la^{\frac{n-2}2}u_0\)d\nu_g \\ &
 =O\( \la^{-n} \int_{M\setminus B_g(\xi,\sqrt\mu)}h\W_{t,\tau}^{2}\left(\la^{\frac{n-2}2}u_0\)^{p-1}d\nu_g\right)+O\(  \la^{-n} \int_{M\setminus B_g(\xi,\sqrt\mu)}h\W_{t,\tau}^{p+1}d\nu_g\)\\&
 +O\(\la^{-\frac{n+2}2} \int_{M\setminus B_g(\xi,\sqrt\mu)} h\W_{t,\tau}^{p } d\nu_g  \)\\ 
 & =\left\{
   \begin{array}{ll} 
   O\(\frac{\mu^\frac{\alpha+2+n}2}{\la^{n-2}}\)& \hbox{if}\ \alpha<n-6\\
   O\(\frac{\mu^\frac{\alpha+2+n}2}{\la^{n-2}}\)|\ln\mu| & \hbox{if}\ \alpha=n-6\\
   O\(\frac{\mu^{n-2}}{\la^{n-2} }\) & \hbox{if}\ \alpha>n-6\\
   \end{array}\right. +O\(\frac{\mu^\frac{\alpha+2+n}2}{\la^n }  \)
  +O\(\frac{\mu^\frac{n+2}2}{\lambda^\frac{n+2}2}\)\\
  &=o\(\frac{\mu^{2+\alpha}}{\lambda^n}  \)\ \hbox{because $\alpha<n-2$}.
  \end{align*}
Collecting the previous computations we get the result.

\medskip
Let us now prove (i) of $(b)$. Observe that 
\begin{align*}
&J_\la(\UU_\la)\\ &=\underbrace{\frac12 \int_M |\nabla_g u_0|_g^2d\nu_g+\frac12\int_M c(n)R_gu_0^2d\nu_g+  \frac1{p+1}\int_M h  u_0 ^{p+1}d\nu_g}_{\hbox{independent of $\mu$ and $\tau$}}\\ &+\la^{-(n-2)}\left[\frac12 \int_M |\nabla_g \W_{t,\tau}|_g^2d\nu_g+\frac12 \int_M c(n)R_g\W_{t,\tau}^2d\nu_g-\frac1{p+1}\int_M\W_{t,\tau}^{p+1}d\nu_g\right] \\
&+\underbrace{\la^{-\frac{n-2}2}\left[\int_M \langle\nabla_g \W_{t,\tau},\nabla_g u_0\rangle+\int_M c(n)R_g \W_{t,\tau}u_0\right]d\nu_g-\la^{-\frac{n-2}2} \int_Mh \W_{t,\tau}u_0^pd\nu_g}_{=0}\\
&-\la^{-(n-2)}\frac1{p+1}\int_M\[\left(\W_{t,\tau}+\la^{\frac{n-2}2}u_0\)^{p+1}-\W_{t,\tau}^{p+1}-\(\la^{\frac{n-2}2}u_0\)^{p+1}\right.\\ &\hskip4truecm\left.- (p+1)\W_{t,\tau}^{p } \(\la^{\frac{n-2}2}u_0\)- (p+1)\W_{t,\tau} \(\la^{\frac{n-2}2}u_0\)^{p }\right]d\nu_g\\
&+\underbrace{\la^{2}\frac1{p+1}\int_Mu_0 ^{p+1}d\nu_g}_{\hbox{independent of $\mu$ and $\tau$}}\\
& 
-\underbrace{\la^{-(n-2)}\int_M\W_{t,\tau}^{p } \(\la^{\frac{n-2}2}u_0\)d\nu_g}_{\hbox{leading term}}\\
&-\la^{-(n-2)}\int_M\W_{t,\tau} \(\la^{\frac{n-2}2}u_0\)^{p }d\nu_g\\
&+\la^{-n}\frac1{p+1}\int_M h\[ \left(\W_{t,\tau}+\la^{\frac{n-2}2}u_0\)^{p+1}- \W_{t,\tau} ^{p+1}-\( \la^{\frac{n-2}2}u_0\)^{p+1}\right.\\ &\hskip4truecm\left.- (p+1)\W_{t,\tau}^p\(\la^{\frac{n-2}2}u_0\) -(p+1)\W_{t,\tau}\(\la^{\frac{n-2}2}u_0\)^{p}\right]d\nu_g\\
&+\underbrace{\la^{-n}\frac1{p+1}\int_M h\W_{t,\tau} ^{p+1}d\nu_g}_{\hbox{leading term}}\\
&+\la^{-n} \int_M h\W_{t,\tau}^p\(\la^{\frac{n-2}2}u_0\)d\nu_g
\end{align*}
Concerning the leading terms, straightforward computations lead to 
$$
\la^{-(n-2)}\int_M\W_{t,\tau}^{p } \(\la^{\frac{n-2}2}u_0\)d\nu_g=u_0(\xi)\frac{\mu^\frac{n-2}2}{\lambda^\frac{n-2}2}\alpha_n^p\int _{\rr^n}
\frac1{(1+|y|^2)^\frac{n+2}2}dy+o\(\frac{\mu^\frac{n-2}2}{\lambda^\frac{n-2}2}\)$$
and
$$
 \la^{-n}\frac1{p+1}\int_M h  \W_{t,\tau} ^{p+1}d\nu_g=
\frac{\mu^{2+\alpha}}{\lambda^n}\frac{\alpha_n^{p+1}}{p+1}\sum\limits_{i=1}^n\int\limits_{\rr^n}a_i \frac{|y_i+\tau_i|^{2+\alpha}}{(1+|y|^2)^n}dy +o\(\frac{\mu^{2+\alpha}}{\lambda^n}\),
$$
 because $\alpha<n-2$.\\
The higher order terms are estimated as follows. 
By \cite{mipive}, we deduce that
\begin{align*}
\la^{-(n-2)}\left[\frac12 \int_M |\nabla_g \W_{t,\tau}|_g^2d\nu_g \right. & \left. + \frac12 \int_M c(n)R_g\W_{t,\tau}^2d\nu_g - \frac{1}{p+1} \int_M\W_{t,\tau}^{p+1}d\nu_g\right]\\ 
& =\la^{-(n-2)}A(n)+\left\{ 
\begin{array}{lr}
\displaystyle O\(\frac\mu\lambda\)& \hbox{if}\ n=3\\
\displaystyle O\(\frac{\mu^2|\ln\mu|}{\lambda^2}\)& \hbox{if}\ n=4\\
\displaystyle O\(\frac{\mu^2}{\lambda^3}\)&\hbox{if}\ n=5
\end{array}\right. \\
& =\la^{-(n-2)}A(n)+ o\(\frac{\mu^\frac{n-2}2}{\lambda^\frac{n-2}2}\),
\end{align*}
where $A(n)$ is a constant that only depends on $n.$
A simple computation shows that
$$\la^{-(n-2)}\int_M\W_{t,\tau} \(\la^{\frac{n-2}2}u_0\)^{p }d\nu_g=O\(\frac{\mu^\frac{n-2}2}{\lambda^\frac{n-2}2}\lambda^2\int_{B_g(\xi,r)} \frac1{\left(d_g(x,\xi)\)^{n-2}} d\nu_g \right)=o\(\frac{\mu^\frac{n-2}2}{\lambda^\frac{n-2}2}\) 
$$
and
\begin{align*}
\la^{-n} \int_M h\W_{t,\tau}^p\(\la^{\frac{n-2}2}u_0\)d\nu_g&=O\(\frac{\mu^\frac{n+2}2}{\lambda^\frac{n+2}2}\int_{B_g(\xi,r)} \frac1{ \left(d_g(x,\xi)\)^{n-\alpha}} d\nu_g \right)\\ 
&=\left\{\begin{array}{ll}
O\(\frac{\mu^{\frac{n+2}2}}{\la^{\frac{n+2}2}} \)& \hbox{if}\ \alpha>0\\
O\(\frac{\mu^{\frac{n+2}2}}{\la^{\frac{n+2}2}}|\ln \mu|\)& \hbox{if}\ \alpha=0 \end{array}\right.\\
&=o\(\frac{\mu^\frac{n-2}2}{\lambda^\frac{n-2}2}\).
\end{align*}

Finally, by using that
$$\left|(a+b)^{p+1}-a^{p+1}-b^{p+1}-(p+1)ab^p-(p+1)a^pb\right|\le c(n)\(a^2b^{p-1}+a^{p-1}b^2\),
$$
which holds if $p\ge 2$ (this is true if $n=3,4,5$) for any $a,b\ge0$, we get
 \begin{align*}
&
 \la^{-(n-2)}\frac1{p+1}\int_M\[\left(\W_{t,\tau}+\la^{\frac{n-2}2}u_0\)^{p+1}-\W_{t,\tau}^{p+1}-\(\la^{\frac{n-2}2}u_0\)^{p+1}\right.\\ &\hskip4truecm\left.- (p+1)\W_{t,\tau}^{p } \(\la^{\frac{n-2}2}u_0\)- (p+1)\W_{t,\tau} \(\la^{\frac{n-2}2}u_0\)^{p }\right]d\nu_g\\
 &=  O\( \la^{-(n-2)} \int_{M}  \W_{t,\tau}^2\left(\la^{\frac{n-2}2}u_0\)^{p-1}d\nu_g \right)+O\( \la^{-(n-2)} \int_{M}  \W_{t,\tau}^{p-1}\left(\la^{\frac{n-2}2}u_0\)^{2}d\nu_g \right)\\
 &=  {O\( \la^{-(n-4)} \int_{M}  \W_{t,\tau}^2d\nu_g\) }+O\(   \int_{M}  \W_{t,\tau}^{p-1} d\nu_g \)\\
 &= {\left\{\begin{array}{ll}
 O\(  \lambda\mu\) &\hbox{if } n=3\\
 O\(  \mu^2|\ln\mu|\)& \hbox{if } n=4\\
 O\(\lambda^{-1}\mu^2 \)& \hbox{if } n=5\\
 \end{array}\right. }
 +
  \left\{\begin{array}{ll}
  O\(  \mu \)& \hbox{if}\ n=3\\
  O\(\mu^2|\ln\mu|\)& \hbox{if}\ n=4\\
  O\(\mu^2 \)& \hbox{if}\ n=5\\
 \end{array}\right. \\
 &  =o\(\frac{\mu^\frac{n-2}2}{\lambda^\frac{n-2}2}\) 
 \end{align*}
 and
 \begin{align*}
&
 \la^{-n}\frac1{p+1}\int_Mh\[\left(\W_{t,\tau}+\la^{\frac{n-2}2}u_0\)^{p+1}-\W_{t,\tau}^{p+1}-\(\la^{\frac{n-2}2}u_0\)^{p+1}\right.\\ &\hskip4truecm\left.- (p+1)\W_{t,\tau}^{p } \(\la^{\frac{n-2}2}u_0\)- (p+1)\W_{t,\tau} \(\la^{\frac{n-2}2}u_0\)^{p }\right]d\nu_g\\
 &=  O\( \la^{-n} \int_{M}h  \W_{t,\tau}^2 \left(\la^{\frac{n-2}2}u_0\)^{p-1}d\nu_g \right)+ O\( \la^{-n} \int_{M}h  \W_{t,\tau}^{p-1} \left(\la^{\frac{n-2}2}u_0\)^{2}d\nu_g \right)\\
 &=  {O\( \la^{-(n-2)}  \int_{M}h  \W_{t,\tau}^2 d\nu_g \)}+ O\( \la^{-2} \int_{M}h  \W_{t,\tau}^{p-1} d\nu_g  \right)
\\  &=  O\left( \la^{-(n-2)} \mu^{n-2}\)+ O\( \la^{-2} \mu^2   \)\\
&=o\(\frac{\mu^\frac{n-2}2}{\lambda^\frac{n-2}2}\) 
.\end{align*} 
Collecting the previous computations we get the result. \end{proof}

\bibliography{PSCreferces}

\begin{bibdiv}
\begin{biblist}

\bib{aubis}{article}{
      author={Aubin, Thierry},
      author={Bismuth, Sophie},
       title={Courbure scalaire prescrite sur les vari\'et\'es riemanniennes
  compactes dans le cas n\'egatif},
        date={1997},
        ISSN={0022-1236},
     journal={J. Funct. Anal.},
      volume={143},
      number={2},
       pages={529\ndash 541},
         url={http://dx.doi.org/10.1006/jfan.1996.3028},
      review={\MR{1428826}},
}

\bib{ah}{article}{
      author={Aubin, Thierry},
      author={Hebey, Emmanuel},
       title={Courbure scalaire prescrite},
        date={1991},
        ISSN={0007-4497},
     journal={Bull. Sci. Math.},
      volume={115},
      number={2},
       pages={125\ndash 131},
      review={\MR{1101020}},
}

\bib{Aubin2}{article}{
      author={Aubin, Thierry},
       title={\'equations diff\'erentielles non lin\'eaires et probl\`eme de
  {Y}amabe concernant la courbure scalaire},
        date={1976},
        ISSN={0021-7824},
     journal={J. Math. Pures Appl. (9)},
      volume={55},
      number={3},
       pages={269\ndash 296},
      review={\MR{0431287}},
}

\bib{Aubin1}{article}{
      author={Aubin, Thierry},
       title={Probl\`emes isop\'erim\'etriques et espaces de {S}obolev},
        date={1976},
        ISSN={0022-040X},
     journal={J. Differential Geometry},
      volume={11},
      number={4},
       pages={573\ndash 598},
         url={http://projecteuclid.org/euclid.jdg/1214433725},
      review={\MR{0448404}},
}

\bib{AubinBook}{book}{
      author={Aubin, Thierry},
       title={Some nonlinear problems in {R}iemannian geometry},
      series={Springer Monographs in Mathematics},
   publisher={Springer-Verlag, Berlin},
        date={1998},
        ISBN={3-540-60752-8},
         url={http://dx.doi.org/10.1007/978-3-662-13006-3},
      review={\MR{1636569}},
}

\bib{bieg}{article}{
      author={Bianchi, Gabriele},
      author={Egnell, Henrik},
       title={A note on the {S}obolev inequality},
        date={1991},
        ISSN={0022-1236},
     journal={J. Funct. Anal.},
      volume={100},
      number={1},
       pages={18\ndash 24},
         url={http://dx.doi.org/10.1016/0022-1236(91)90099-Q},
      review={\MR{1124290}},
}

\bib{bgs}{article}{
      author={Borer, Franziska},
      author={Galimberti, Luca},
      author={Struwe, Michael},
       title={``{L}arge'' conformal metrics of prescribed {G}auss curvature on
  surfaces of higher genus},
        date={2015},
        ISSN={0010-2571},
     journal={Comment. Math. Helv.},
      volume={90},
      number={2},
       pages={407\ndash 428},
         url={http://dx.doi.org/10.4171/CMH/358},
      review={\MR{3351750}},
}

\bib{bis}{article}{
      author={Bismuth, Sophie},
       title={Courbure scalaire prescrite sur une vari\'et\'e riemannienne
  {$C^\infty$} compacte dans le cas nul},
        date={1998},
        ISSN={0021-7824},
     journal={J. Math. Pures Appl. (9)},
      volume={77},
      number={7},
       pages={667\ndash 695},
         url={http://dx.doi.org/10.1016/S0021-7824(98)80004-4},
      review={\MR{1645069}},
}

\bib{blr}{article}{
      author={Bahri, Abbas},
      author={Li, Yanyan},
      author={Rey, Olivier},
       title={On a variational problem with lack of compactness: the
  topological effect of the critical points at infinity},
        date={1995},
        ISSN={0944-2669},
     journal={Calc. Var. Partial Differential Equations},
      volume={3},
      number={1},
       pages={67\ndash 93},
         url={http://dx.doi.org/10.1007/BF01190892},
      review={\MR{1384837}},
}

\bib{delpino}{article}{
      author={del Pino, Manuel~A.},
       title={Positive solutions of a semilinear elliptic equation on a compact
  manifold},
        date={1994},
        ISSN={0362-546X},
     journal={Nonlinear Anal.},
      volume={22},
      number={11},
       pages={1423\ndash 1430},
         url={http://dx.doi.org/10.1016/0362-546X(94)90121-X},
      review={\MR{1280207}},
}

\bib{dp1}{article}{
      author={del Pino, Manuel},
      author={Felmer, Patricio},
      author={Musso, Monica},
       title={Two-bubble solutions in the super-critical {B}ahri-{C}oron's
  problem},
        date={2003},
        ISSN={0944-2669},
     journal={Calc. Var. Partial Differential Equations},
      volume={16},
      number={2},
       pages={113\ndash 145},
         url={http://dx.doi.org/10.1007/s005260100142},
      review={\MR{1956850}},
}

\bib{dp2}{article}{
      author={del Pino, Manuel},
      author={Musso, Monica},
      author={Pistoia, Angela},
       title={Super-critical boundary bubbling in a semilinear {N}eumann
  problem},
        date={2005},
        ISSN={0294-1449},
     journal={Ann. Inst. H. Poincar\'e Anal. Non Lin\'eaire},
      volume={22},
      number={1},
       pages={45\ndash 82},
         url={http://dx.doi.org/10.1016/j.anihpc.2004.05.001},
      review={\MR{2114411}},
}

\bib{dp3}{article}{
      author={del Pino, Manuel},
      author={Musso, Monica},
      author={Rom\'an, Carlos},
      author={Wei, Juncheng},
       title={Interior bubbling solutions for the critical {L}in-{N}i-{T}akagi
  problem in dimension $3$},
        date={2016},
     journal={Accepted for publication in J. Anal. Math.},
}

\bib{pinrom}{article}{
      author={del Pino, Manuel},
      author={Rom\'an, Carlos},
       title={Large conformal metrics with prescribed sign-changing {G}auss
  curvature},
        date={2015},
        ISSN={0944-2669},
     journal={Calc. Var. Partial Differential Equations},
      volume={54},
      number={1},
       pages={763\ndash 789},
         url={http://dx.doi.org/10.1007/s00526-014-0805-y},
      review={\MR{3385180}},
}

\bib{espi}{article}{
      author={Esposito, Pierpaolo},
      author={Pistoia, Angela},
       title={Blowing-up solutions for the {Y}amabe equation},
        date={2014},
        ISSN={0032-5155},
     journal={Port. Math.},
      volume={71},
      number={3-4},
       pages={249\ndash 276},
         url={http://dx.doi.org/10.4171/PM/1952},
      review={\MR{3298464}},
}

\bib{espive1}{article}{
      author={Esposito, Pierpaolo},
      author={Pistoia, Angela},
      author={V\'etois, J\'er\^ome},
       title={The effect of linear perturbations on the {Y}amabe problem},
        date={2014},
        ISSN={0025-5831},
     journal={Math. Ann.},
      volume={358},
      number={1-2},
       pages={511\ndash 560},
         url={http://dx.doi.org/10.1007/s00208-013-0971-9},
      review={\MR{3158007}},
}

\bib{escobar}{article}{
      author={Escobar, Jos\'e~F.},
      author={Schoen, Richard~M.},
       title={Conformal metrics with prescribed scalar curvature},
        date={1986},
        ISSN={0020-9910},
     journal={Invent. Math.},
      volume={86},
      number={2},
       pages={243\ndash 254},
         url={http://dx.doi.org/10.1007/BF01389071},
      review={\MR{856845}},
}

\bib{hv1}{article}{
      author={Hebey, Emmanuel},
      author={Vaugon, Michel},
       title={Meilleures constantes dans le th\'eor\`eme d'inclusion de
  {S}obolev et multiplicit\'e pour les probl\`emes de {N}irenberg et {Y}amabe},
        date={1992},
        ISSN={0022-2518},
     journal={Indiana Univ. Math. J.},
      volume={41},
      number={2},
       pages={377\ndash 407},
         url={http://dx.doi.org/10.1512/iumj.1992.41.41021},
      review={\MR{1183349}},
}

\bib{hv2}{article}{
      author={Hebey, Emmanuel},
      author={Vaugon, Michel},
       title={Le probl\`eme de {Y}amabe \'equivariant},
        date={1993},
        ISSN={0007-4497},
     journal={Bull. Sci. Math.},
      volume={117},
      number={2},
       pages={241\ndash 286},
      review={\MR{1216009}},
}

\bib{kawa}{article}{
      author={Kazdan, Jerry~L.},
      author={Warner, F.~W.},
       title={Scalar curvature and conformal deformation of {R}iemannian
  structure},
        date={1975},
        ISSN={0022-040X},
     journal={J. Differential Geometry},
      volume={10},
       pages={113\ndash 134},
         url={http://projecteuclid.org/euclid.jdg/1214432678},
      review={\MR{0365409}},
}

\bib{yyl2}{article}{
      author={Li, Yan~Yan},
       title={Prescribing scalar curvature on {$S^n$} and related problems.
  {I}},
        date={1995},
        ISSN={0022-0396},
     journal={J. Differential Equations},
      volume={120},
      number={2},
       pages={319\ndash 410},
         url={http://dx.doi.org/10.1006/jdeq.1995.1115},
      review={\MR{1347349}},
}

\bib{yyl3}{article}{
      author={Li, Yanyan},
       title={Prescribing scalar curvature on {$S^n$} and related problems.
  {II}. {E}xistence and compactness},
        date={1996},
        ISSN={0010-3640},
     journal={Comm. Pure Appl. Math.},
      volume={49},
      number={6},
       pages={541\ndash 597},
  url={http://dx.doi.org/10.1002/(SICI)1097-0312(199606)49:6<541::AID-CPA1>3.0.CO;2-A},
      review={\MR{1383201}},
}

\bib{yyl}{article}{
      author={Li, Yan~Yan},
       title={On a singularly perturbed equation with {N}eumann boundary
  condition},
        date={1998},
        ISSN={0360-5302},
     journal={Comm. Partial Differential Equations},
      volume={23},
      number={3-4},
       pages={487\ndash 545},
         url={http://dx.doi.org/10.1080/03605309808821354},
      review={\MR{1620632}},
}

\bib{mipive}{article}{
      author={Micheletti, Anna~Maria},
      author={Pistoia, Angela},
      author={V\'etois, J\'er\^ome},
       title={Blow-up solutions for asymptotically critical elliptic equations
  on {R}iemannian manifolds},
        date={2009},
        ISSN={0022-2518},
     journal={Indiana Univ. Math. J.},
      volume={58},
      number={4},
       pages={1719\ndash 1746},
         url={http://dx.doi.org/10.1512/iumj.2009.58.3633},
      review={\MR{2542977}},
}

\bib{ouy}{article}{
      author={Ouyang, Tiancheng},
       title={On the positive solutions of semilinear equations {$\Delta
  u+\lambda u+hu^p=0$} on compact manifolds. {II}},
        date={1991},
        ISSN={0022-2518},
     journal={Indiana Univ. Math. J.},
      volume={40},
      number={3},
       pages={1083\ndash 1141},
         url={http://dx.doi.org/10.1512/iumj.1991.40.40049},
      review={\MR{1129343}},
}

\bib{rau1}{article}{
      author={Rauzy, Antoine},
       title={Courbures scalaires des vari\'et\'es d'invariant conforme
  n\'egatif},
        date={1995},
        ISSN={0002-9947},
     journal={Trans. Amer. Math. Soc.},
      volume={347},
      number={12},
       pages={4729\ndash 4745},
         url={http://dx.doi.org/10.2307/2155060},
      review={\MR{1321588}},
}

\bib{rau2}{article}{
      author={Rauzy, Antoine},
       title={Multiplicit\'e pour un probl\`eme de courbure scalaire
  prescrite},
        date={1996},
        ISSN={0007-4497},
     journal={Bull. Sci. Math.},
      volume={120},
      number={2},
       pages={153\ndash 194},
      review={\MR{1387420}},
}

\bib{rove}{article}{
      author={Robert, Fr\'ed\'eric},
      author={V\'etois, J\'er\^ome},
       title={Sign-changing blow-up for scalar curvature type equations},
        date={2013},
        ISSN={0360-5302},
     journal={Comm. Partial Differential Equations},
      volume={38},
      number={8},
       pages={1437\ndash 1465},
         url={http://dx.doi.org/10.1080/03605302.2012.745552},
      review={\MR{3169751}},
}

\bib{Sch}{article}{
      author={Schoen, Richard},
       title={Conformal deformation of a {R}iemannian metric to constant scalar
  curvature},
        date={1984},
        ISSN={0022-040X},
     journal={J. Differential Geom.},
      volume={20},
      number={2},
       pages={479\ndash 495},
         url={http://projecteuclid.org/euclid.jdg/1214439291},
      review={\MR{788292}},
}

\bib{struwe}{book}{
      author={Struwe, Michael},
       title={Variational methods},
     edition={Fourth},
      series={Ergebnisse der Mathematik und ihrer Grenzgebiete. 3. Folge. A
  Series of Modern Surveys in Mathematics [Results in Mathematics and Related
  Areas. 3rd Series. A Series of Modern Surveys in Mathematics]},
   publisher={Springer-Verlag, Berlin},
        date={2008},
      volume={34},
        ISBN={978-3-540-74012-4},
        note={Applications to nonlinear partial differential equations and
  Hamiltonian systems},
      review={\MR{2431434}},
}

\bib{Talenti}{article}{
      author={Talenti, Giorgio},
       title={Best constant in {S}obolev inequality},
        date={1976},
        ISSN={0003-4622},
     journal={Ann. Mat. Pura Appl. (4)},
      volume={110},
       pages={353\ndash 372},
         url={http://dx.doi.org/10.1007/BF02418013},
      review={\MR{0463908}},
}

\bib{tru}{article}{
      author={Trudinger, Neil~S.},
       title={Remarks concerning the conformal deformation of {R}iemannian
  structures on compact manifolds},
        date={1968},
     journal={Ann. Scuola Norm. Sup. Pisa (3)},
      volume={22},
       pages={265\ndash 274},
      review={\MR{0240748}},
}

\bib{vave}{article}{
      author={V\'azquez, Juan~Luis},
      author={V\'eron, Laurent},
       title={Solutions positives d'\'equations elliptiques semi-lin\'eaires
  sur des vari\'et\'es riemanniennes compactes},
        date={1991},
        ISSN={0764-4442},
     journal={C. R. Acad. Sci. Paris S\'er. I Math.},
      volume={312},
      number={11},
       pages={811\ndash 815},
      review={\MR{1108497}},
}

\bib{yam}{article}{
      author={Yamabe, Hidehiko},
       title={On a deformation of {R}iemannian structures on compact
  manifolds},
        date={1960},
     journal={Osaka Math. J.},
      volume={12},
       pages={21\ndash 37},
      review={\MR{0125546}},
}

\end{biblist}
\end{bibdiv}
\end{document}